\documentclass{article}

\usepackage[utf8]{inputenc} 
\usepackage[T1]{fontenc}    
\usepackage{hyperref}       
\usepackage{url}            
\usepackage{booktabs}       
\usepackage{amsfonts}       
\usepackage{nicefrac}       
\usepackage{microtype}      
\usepackage{xcolor}         

\usepackage{algorithm}
\usepackage{algpseudocode}
\usepackage{placeins}
\usepackage{authblk}
\usepackage[margin=1in]{geometry}
\usepackage{tikz}
\usepackage{multirow}
\usepackage{pgffor}
\usepackage{enumitem}
\usepackage{graphicx}
\usetikzlibrary{arrows}
\usepackage{subcaption}
\usepackage{amsmath, amsthm, amssymb}
\newtheorem{Theorem}{Theorem}
\newtheorem{Proposition}[Theorem]{Proposition}
\newtheorem{Remark}[Theorem]{Remark}

\newtheorem{Corollary}[Theorem]{Corollary}
\newtheorem{Lemma}[Theorem]{Lemma}
\newtheorem{Example}[Theorem]{Example}

\theoremstyle{definition}
\newtheorem{Definition}[Theorem]{Definition}

\newcommand{\define}[1]{\textbf{#1}}

\newcommand{\dgh}{d_{\mathrm{GH}}}
\newcommand{\corr}{\mathcal{R}}
\newcommand{\dsgh}{d_{\mathrm{srGH}}}
\newcommand{\symdsgh}{\widehat{d}_\mathrm{srGH}}
\newcommand{\semcorr}{\mathcal{SR}}
\newcommand{\modgh}{d_\mathrm{mGH}}
\newcommand{\dis}{\mathrm{dis}}
\newcommand{\codis}{\mathrm{codis}}
\newcommand{\coup}{\mathcal{C}}
\newcommand{\semcoup}{\mathcal{SC}}
\newcommand{\dgwp}{d_{\mathrm{GW},p}}
\newcommand{\dsgwp}{d_{\mathrm{srGW},p}}
\newcommand{\symdsgwp}{\widehat{d}_{\mathrm{srGW},p}}
\newcommand{\dhaus}{d_\mathrm{H}}
\newcommand{\dsh}{d_{\mathrm{srH}}}
\newcommand{\R}{\mathbb{R}}

\title{Generalized Dimension Reduction Using Semi-Relaxed Gromov-Wasserstein Distance}

\author[1]{Ranthony A. Clark}
\author[2]{Tom Needham}
\author[3]{Thomas Weighill}
\affil[1]{Department of Mathematics,  Duke University,  Durham NC, 27710} 
\affil[2]{Department of Mathematics,  Florida State University,  Tallahassee FL, 32304}
\affil[3]{Department of Mathematics and Statistics,  University of North Carolina at Greensboro,  Greensboro NC, 27412}

\begin{document}

\maketitle

\begin{abstract}
Dimension reduction techniques typically seek an embedding of a high-dimensional point cloud into a low-dimensional Euclidean space which optimally preserves the geometry of the input data. Based on expert knowledge, one may instead wish to embed the data into some other manifold or metric space in order to better reflect the geometry or topology of the point cloud. We propose a general method for manifold-valued multidimensional scaling based on concepts from optimal transport. In particular, we establish theoretical connections between the recently introduced semi-relaxed Gromov-Wasserstein (srGW) framework and multidimensional scaling by solving the Monge problem in this setting. We also derive novel connections between srGW distance and Gromov-Hausdorff distance. We apply our computational framework to analyze ensembles of political redistricting plans for states with two Congressional districts, achieving an effective visualization of the ensemble as a distribution on a circle which can be used to characterize typical neutral plans, and to flag outliers.
\end{abstract}

\section{Introduction}\label{sec:intro}

Dimension reduction is a fundamental task in unsupervised learning and is frequently a first step in the data exploration pipeline. Typically, dimension reduction is framed as the process of determining an embedding of a finite metric space $(X,d)$ into a low-dimensional Euclidean space $\R^m$ which optimally preserves the geometry of the input data. As a concrete example, for $X = \{x_1,\ldots,x_n\}$, the \define{metric multidimensional scaling (MDS) problem} seeks  a point cloud $\mathrm{MDS}_m(X) \subset \R^m$ satisfying
\begin{equation}\label{eqn:MDS}
\mathrm{MDS}_m(X) \in \underset{y_1,\ldots,y_n}{\mathrm{argmin}} \sum_{i,j=1}^n (d(x_i,x_j) - \|y_i - y_j\| )^2.
\end{equation}

Based on prior knowledge, it may be natural to theorize that a dataset $X$ is noisily sampled from a distribution on a specific low-dimensional manifold in the ambient data space, and the practitioner may wish for this structure to be reflected in the dimension reduction process---indeed, this was the case for the geospatial data studied in Section \ref{sec:redistricting}, which catalyzed the ideas in this paper. A simple observation is that the MDS problem \eqref{eqn:MDS}  still makes sense if $\R^m$ is replaced with some other low-dimensional Riemannian manifold, and the main goal of this paper is to develop a theoretical and computational framework for this manifold-valued variant of MDS. Another observation is that the objective of the MDS problem is very similar to that of the Gromov-Wasserstein (GW) distance from optimal transport theory~\cite{memoli2011gromov}, and our main theoretical result makes this connection precise. 

\paragraph{Contributions.} The contributions of this paper are centered on theoretical results which tie together several ideas from classical and more recent literature. The impetus for the paper was the problem of visualizing complex redistricting datasets; our solution provides an extended computational example that illustrates how our framework can be used to gain important insights into non-Euclidean data. More specifically, our main contributions are:

\begin{itemize}[leftmargin=*]
    \item We extend the formulation of semi-relaxed Gromov-Wasserstein distance introduced in~\cite{vincent2022semi} to a 1-parameter family of $L^p$-type distances defined for general metric measure spaces (including continuous spaces). We then show in Corollary \ref{cor:connection_to_MDS} that the semi-relaxed GW problem generalizes the MDS problem \eqref{eqn:MDS} in several ways. This is based on Theorem \ref{thm:monge}, which shows that the semi-relaxed GW distance is realized by a Monge map in very general situations, thus adding to the growing recent literature on the existence of Monge maps in the GW framework (see Remark \ref{remark:monge}).

    \item We develop and unify the theory of generalized MDS and semi-relaxed GW distances by exhibiting connections to variants of the Gromov-Hausdorff (GH) distance that have appeared previously in the literature. Theorem \ref{thm:equivalent_mGH} and Theorem \ref{thm:srgh_equals_srgw} together show that a symmetrized version of the $p=\infty$ semi-relaxed GW distance is equal to the modified GH distance of~\cite{memoli2012some}, which, in turn, appeared in classical work on generalized MDS~\cite{bronstein2006generalized}.

    \item By drawing the connection between the MDS and srGW problems, we are able to design an efficient algorithm (SRGW+GD) for computing MDS embeddings into manifolds, consisting of an initial discretized optimal transport computation followed by a gradient descent stage. This allows general target spaces, and produces significantly better embeddings (quantitatively and qualitatively) than a naive gradient descent algorithm (Table \ref{tab:distortiontable} and Figure \ref{fig:scatterplots}). 
    
    \item Using experiments on the MNIST dataset, we show that SRGW+GD matches or outperforms SMACOF MDS for Euclidean target spaces. We demonstrate its effectiveness for embeddings into spheres on a dataset of rotated MNIST images, and a set of GPS coordinates of cities. 
    
    \item In the final section, we apply SRGW+GD to ensembles of political districting plans for low-population states, achieving a natural and effective visualization of each ensemble as a distribution on a circle. These visualizations lead to pertinent insights into the distribution of possible districting plans for each state.
\end{itemize}

\section{Semi-Relaxed Gromov-Wasserstein Distance and Multidimensional Scaling}\label{sec:gwdistances}

\paragraph{Gromov-Wasserstein Distances.}
Let $(X,d_X,\mu_X)$ and $(Y,d_Y,\mu_Y)$ be \define{metric measure spaces (mm-spaces)}; that is, $(X,d_X)$ is a metric space, which we assume to be complete and separable, and $\mu_X$ is a Borel probability measure on $X$. We will sometimes abuse notation and simply write $X$ for $(X,d_X,\mu_X)$ when it is clear that there is an associated choice of metric and measure.

A \define{coupling} of $\mu_X$ and $\mu_Y$ is a Borel probability measure $\gamma$ on $X \times Y$ with marginals equal to $\mu_X$ and $\mu_Y$, respectively. Writing this symbolically, $(\pi_1)_\# \gamma = \mu_X$ and $(\pi_2)_\# \gamma = \mu_Y$, where $\pi_1:X \times Y \to X$ and $\pi_2:X \times Y \to Y$ are the coordinate projection maps and $(\pi_1)_\# \gamma$ denotes the pushforward measure. We denote the set of all measure couplings between $\mu_X$ and $\mu_Y$ as $\coup(\mu_X,\mu_Y)$. For $p \in [1,\infty)$, the \define{Gromov-Wasserstein (GW) $p$-distance}~\cite{10.2312:SPBG:SPBG07:081-090,memoli2011gromov} is 
\[
\dgwp(X,Y) = \inf_{\gamma \in \coup(\mu_X,\mu_Y)} \dis_p(\gamma),
\]
where the \define{$p$-distortion} of $\gamma$, $\dis_p(\gamma)$, is given by

\begin{align}
& \dis_p(\gamma) = \frac{1}{2} \|\Gamma_{X,Y}\|_{L^p(\gamma \otimes \gamma)} \label{eqn:p_distortion} \\ 
& = \frac{1}{2} \left( \iint | \Gamma_{X,Y}(x,y,x',y')|^p \, d\gamma(x,y) \, d\gamma(x', y') \right)^{1/p} \nonumber 
\end{align}
with $\Gamma_{X,Y}(x,y,x',y') = d_X(x,x') - d_Y(y,y')$. 

This extends to the $p = \infty$ case, where the distortion is equal to 
\begin{align}
\dis_\infty(\gamma) 
& = \frac{1}{2} \|\Gamma_{X,Y}\|_{L^\infty(\mathrm{supp}(\gamma) \times \mathrm{supp}(\gamma))} \nonumber \\ 
& = \frac{1}{2} \sup_{(x,y),(x',y') \in \mathrm{supp}(\gamma)} |\Gamma_{X,Y}(x,y,x',y')| \nonumber,
\end{align}
with $\mathrm{supp}(\gamma)$ denoting the support of the measure $\gamma$.

The GW distance was introduced in~\cite{10.2312:SPBG:SPBG07:081-090}, where it was shown that $\dgwp$ defines a metric on the space of measure-preserving isometry classes of fully supported compact mm-spaces.

\paragraph{Semi-Relaxed GW Distances.} GW distances have become a popular tool in the machine learning community, due to their ability to compare distinct data types~\cite{peyre2016gromov,chowdhury2021generalized,
xu2019scalable}. Many variants of GW distances have been introduced in recent years~\cite{scetbon2022linear,chowdhury2021quantized,chowdhury2023hypergraph,titouan2020co,vayer2020fused}. Of particular interest here is the \emph{semi-relaxed Gromov-Wasserstein (srGW) distance} of~\cite{vincent2022semi}. There, the $p=2$ version was defined for finite spaces, and we give a natural generalization here.

\begin{Definition}[Semi-Relaxed Gromov-Wasserstein Distance]
    Let $X = (X,d_X,\mu_X)$ and $Y = (Y,d_Y,\mu_Y)$ be mm-spaces and let $p \in [1,\infty]$. A \define{semi-coupling} of $\mu_X$ and $Y$ is a Borel probability measure $\gamma$ on $X \times Y$ such that $(\pi_1)_\# \gamma = \mu_X$ (note that this definition doesn't depend on the measure on $Y$). The set of semi-couplings will be denoted $\semcoup(\mu_X,Y)$. The \define{semi-relaxed Gromov-Wasserstein (srGW) $p$-distance} is 
    \begin{equation}\label{eqn:srGW}
    \dsgwp(X,Y) = \inf_{\gamma \in \semcoup(\mu_X,Y)} \dis_p(\gamma),
    \end{equation}
    where the $p$-distortion is as defined in \eqref{eqn:p_distortion}.    
\end{Definition}

We call $\dsgwp$ a ``distance" in an informal sense. It is clearly asymmetric and, in fact, we will show in Section~\ref{sec:gromov-type_distances} that it satisfies the triangle inequality if and only if $p=\infty$. For now, we consider \eqref{eqn:srGW} as an interesting optimization problem. We show below that it is closely related to MDS \eqref{eqn:MDS}.

\paragraph{Monge Maps and Generalized MDS.}

Let $(X,d_X,\mu_X)$ be a mm-space and $Y$ a Polish space. Given a measurable function $f:X \to Y$, we define the \define{semi-coupling induced by $f$} to be the measure $\mu_f$ on $X\times Y$ given by
$\mu_f := (\mathrm{id}_X \times f)_\# \mu_X$,
where $\mathrm{id}_X \times f: X \to X \times Y$ is the function $x \mapsto (x,f(x))$. In the case that $X$ is finite, this measure is given explicitly by
\[
\mu_f = \sum_{x \in X} \mu_X(x) \delta_{(x, f(x))}.
\]
Here, and throughout the rest of the paper, we use $\delta_z$ to denote the Dirac mass at a point $z \in Z$. 

We recall that a metric space is called \define{proper} if its closed and bounded sets are compact. An action of a group $G$ on a metric space $X$ is said to be \define{cocompact} if there exists a compact $K \subseteq X$ such that $X$ is covered by translates of $K$ under the $G$-action. The following is our main result.

\begin{Theorem}[Existence of Monge Maps]\label{thm:monge}
Let $(X,d_X,\mu_X)$ be a metric measure space with $X$ finite and $\mu_X$ fully supported and let $(Y,d_Y)$ be a proper metric space with a cocompact action by isometries by some group $G$. Then for any $p \in [1,\infty]$, there exists a function $f: X \to Y$ such that 
\[
\dsgwp(X,Y) = \dis_p(\mu_f).
\]
Moreover, if $p < \infty$, then any semi-coupling with the same distortion as $\mu_f$ is induced by a function.
\end{Theorem}

The proof is included in the Appendix. The main idea is that, from any initial coupling, one can construct a new one of the form $\mu_f$ with lower distortion via disintegration of the initial coupling.

\begin{Remark}
    The theorem applies when $Y = \mathbb{R}^n$, endowed with Euclidean distance (a main motivating example), but also to a large class of spaces including compact metric spaces (e.g.~closed Riemannian manifolds) or infinite binary trees. Regarding the source space $X$, our proof requires $X$ to be finite; removing this assumption, even in the Euclidean case, seems difficult -- see e.g.~\cite{murray2024probabilistic}. 
\end{Remark}

\begin{Remark}[The Monge Problem]\label{remark:monge}
    We will refer to the map $f$ in Theorem \ref{thm:monge} as a \define{Monge map}, in reference to the original formulation of optimal transport, due to Monge, where optimization was performed over measure-preserving maps, rather than over couplings (see~\cite{villani2021topics}). The \define{Monge problem} in OT theory is to determine conditions under which the Wasserstein distance is realized by a measure-preserving map; the problem is now well-understood, with several general results, e.g.,~\cite{brenier1987decomposition} and~\cite[Theorem 10.41]{villani2009optimal}. The Monge problem in the GW setting is an active area of current research. The state-of-the-art results in this direction appear in~\cite{dumont2024existence}, where the problem is solved for variants of the $p=2$ GW distance between measures on Euclidean spaces with density. Other recent results for (variants of) the Monge problem for GW distance between more restrictive subclasses of mm-spaces appear in~\cite{vayer2020contribution,sturm2023space,beinert2023assignment,salmona2022gromov,memoli2024comparison}.
\end{Remark}

\begin{Example}
    Without the assumptions on $Y$ made in Theorem \ref{thm:monge}, there may be no Monge map from $X$ to $Y$. Indeed, let $X = \{0,1\} \subseteq \mathbb{R}$ and let $Y = \{(0,2n) \mid n \in \mathbb{N} \} \cup \{(1+2^{-n},2n) \mid n \in \mathbb{N} \}$. Since $X$ does not isometrically embed into $Y$, there is no zero distortion map from $X$ to $Y$. However, the maps $f_n: X \to Y$ given by $f(0) = (0,2n),\ f(1) = (1+2^{-n}, 2n)$ have arbitrarily low distortion.
\end{Example}

We have the following corollary, showing that the srGW problem generalizes the MDS problem.

\begin{Corollary}\label{cor:connection_to_MDS}
    Let $(X,d)$ be a finite metric space with $X = \{x_1,\ldots,x_n\}$ and let $\mu$ be uniform measure on $X$. Let $f:X \to \R^m$ be a function which realizes $d_{\mathrm{srGW},2}(X,\R^m)$. Then the point cloud $f(x_1),\ldots,f(x_n) \in \R^m$ is a solution of $\mathrm{MDS}_m(X)$. 
\end{Corollary}

\paragraph{Related Work.} The semi-relaxed GW problem was first studied in~\cite{vincent2022semi}, where it was applied to graph machine learning problems such as dictionary learning and graph completion. The followup paper~\cite{van2024distributional} considers theoretical aspects of srGW and, in particular, connections to dimension reduction. Their result~\cite[Theorem 3.2]{van2024distributional} is analogous to Corollary~\ref{cor:connection_to_MDS} and says that spectral methods of dimensional reduction can be realized as solutions to semi-relaxed GW problems. This covers, for example, classical multidimensional scaling (as opposed to metric MDS, studied here). These spectral methods are only applicable to Euclidean (or perhaps hyperbolic) embedding spaces, so that the target applications of~\cite{van2024distributional} and this paper are fairly disjoint. Our theoretical results are related to and complementary with those of~\cite{van2024distributional}, but neither paper generalizes the other.

The problem of dimension reduction into a non-Euclidean space has been well-studied in the topological data analysis literature. These approaches are fundamentally different than the one used here; they rely on specific constructions of algebraic topology and are therefore only suited to embedding in specific classes of spaces such as circles~\cite{de2009persistent,paik2023circular}, projective spaces~\cite{perea2018multiscale}, or lens spaces~\cite{polanco2019lens}. More importantly, the results of these algorithms are qualitatively different than ours, as they are based on persistent (co)homology rather than on geometry preservation. We compare our results to one of these methods, circular coordinates, in Section \ref{sec:experiments}.

The Gromov-Hausdorff distance is another useful tool for shape comparison and analysis where shapes can be modeled as (compact) metric spaces, versus the Gromov-Wasserstein setting which considers the distributional properties of a sample via metric measure spaces. It is used in areas such as manifold learning, computer vision, computational geometry, and topological data analysis which emphasize the geometry and topological properties of data. Our work is closely related to the influential paper~\cite{bronstein2006generalized}, which approaches the problem of partial surface matching through a certain ``semi-relaxed" version of Gromov-Hausdorff distance. This connection inspired the work in the next section, which derives a precise relationship between the version of Gromov-Hausdorff distance considered in~\cite{bronstein2006generalized} and the srGW distance introduced of~\cite{vincent2022semi}. 

\section{Connections to Gromov-Hausdorff Distance}\label{sec:gromov-type_distances}

\paragraph{Semi-Relaxed Gromov-Hausdorff Distance.} Let $X = (X,d_X)$ and $Y = (Y,d_Y)$ be metric spaces. Recall that the \define{Gromov-Hausdorff (GH) distance} between $X$ and $Y$ is defined by 
\begin{equation}\label{eqn:gromov_hausdorff}
\dgh(X,Y) = \inf_{R \in \corr(X,Y)} \dis(R),
\end{equation}
where $\corr(X,Y)$ is the set of \define{correspondences} between $X$ and $Y$---that is, the set of relations $R \subset X \times Y$ such that the coordinate projection maps take $R$ surjectively onto each component---and $\dis(R)$ is the \define{metric distortion} of the correspondence $R$, defined by
\begin{equation}\label{eqn:metric_distortion}
\dis(R) = \frac{1}{2} \sup_{(x,y),(x',y') \in R} |d_X(x,x') - d_Y(y,y')|.
\end{equation}

Inspired by the semi-relaxed version of Gromov-Wasserstein distance considered above, we now define a semi-relaxed version of Gromov-Hausdorff distance.

\begin{Definition}[Semi-Relaxed Gromov-Hausdorff Distance]\label{def:semi_relaxed_GH}
    Let $X$ and $Y$ be metric spaces. A relation $R \subset X \times Y$ is called a \define{semi-correspondence} if the coordinate projection to $X$ takes $R$ surjectively onto $X$ (but we put no such condition on the projection map to $Y$). Let $\semcorr(X,Y)$ denote the set of semi-correspondences. 
    The \define{semi-relaxed Gromov-Hausdorff (srGH) distance} is 
    \[
    \dsgh(X,Y) = \inf_{R \in \semcorr(X,Y)} \dis(R),
    \]
    with $\dis(R)$ defined as in \eqref{eqn:metric_distortion}. This is symmetrized as
    \[
    \symdsgh(X,Y) = \max \{\dsgh(X,Y),\dsgh(Y,X)\}.
    \]   
\end{Definition}

\paragraph{Equivalence to Modified Gromov-Hausdorff Distance.} We will now show that the symmetrized srGH distance is a reformulation of a distance which has already appeared in the literature. Recall (see, e.g.,  \cite{burago2022course}) that the GH distance can be expressed as
\begin{equation}\label{eqn:GH_reformulation}
\dgh(X,Y) =  \inf_{f,g} \max \{\dis(f), \dis(g), \codis(f,g)\},
\end{equation}
where the infimum is over (not necessarily continuous) functions $f:X \to Y$ and $g:Y \to X$, the \define{function distortion} $\dis(f)$ is defined by
\[
\dis(f) = \frac{1}{2} \sup_{x,x' \in X} |d_X(x,x') - d_Y(f(x),f(x'))|,
\]
with $\dis(g)$ defined similarly, and where the \define{codistortion} $\codis(f,g)$ is defined by
\[
\codis(f,g) = \frac{1}{2} \sup_{x \in X, y \in Y} |d_X(x,g(y)) - d_Y(y,f(x))|.
\]
This formulation leads to a natural modification, where the maps $f$ and $g$ are decoupled by dropping the codistortion term in \eqref{eqn:GH_reformulation}. The \define{modified Gromov-Hausdorff distance} is 
\begin{align*}
\modgh(X,Y) & = \inf_{f,g} \max\{\dis(f),\dis(g)\} \\ & = \max \left\{\inf_{f:X \to Y} \dis(f), \inf_{g:Y \to X} \dis(g)\right\}.
\end{align*}
This distance was introduced in \cite{memoli2012some}, where it was shown to be a metric on the space of isometry classes of compact metric spaces~\cite[Theorem 4.1]{memoli2012some}. Our next main result shows that it is the same as the symmetrized semi-relaxed GH distance.

\begin{Theorem}[Equivalence of Gromov-Hausdorff Distances]\label{thm:equivalent_mGH}
    The symmetrized semi-relaxed Gromov-Hausdorff distance $\symdsgh$ is equal to the modified Gromov-Hausdorff distance $\modgh$.
\end{Theorem}

The proof is given in the Appendix. We also provide an additional characterization of semi-relaxed GH distance in terms of isometric embeddings in the Appendix.

\paragraph{Connection to Semi-Relaxed Gromov-Wasserstein Distance.} There is an apparent relationship between the $p=\infty$ version of (semi-relaxed) Gromov-Wasserstein distance and (semi-relaxed) Gromov-Hausdorff distance. Indeed, given mm-spaces $(X,d_X,\mu_X)$ and $(Y,d_Y,\mu_Y)$ with fully-supported measures, any coupling $\gamma \in \coup(\mu_X,\mu_Y)$ induces a correspondence $\mathrm{supp}(\gamma)$, and it follows that 
\begin{equation}\label{eqn:GH_lower_bound}
\dgh(X,Y) \leq d_{\mathrm{GW},\infty}(X,Y),
\end{equation}
where the quantity on the left is understood as the Gromov-Hausdorff distance between the underlying metric spaces. However, the inequality \eqref{eqn:GH_lower_bound} is not an equality, in general---see~\cite[Theorem 5.1]{memoli2011gromov}. The following result shows that equality does hold in the semi-relaxed setting. We define the \define{symmetrized srGW distance} $\hat{d}_{\mathrm{srGW},p}$ by 
\[
\widehat{d}_{\mathrm{srGW},p}(X,Y) = \max\{\dsgwp(X,Y), \dsgwp(Y,X)\}.
\]

\begin{Theorem}[Equivalence of srGW and srGH]\label{thm:srgh_equals_srgw}
    Let $X$ and $Y$ be mm-spaces such that $\mu_X$ has full support. Then 
    \[
    \widehat{d}_{\mathrm{srGW},\infty}(X,Y) = \widehat{d}_{\mathrm{srGH}}(X,Y) = \modgh(X,Y).
    \]
    Moreover, $\widehat{d}_{\mathrm{srGW},p}$ defines a metric on the space of isometry classes of compact metric spaces which is topologically equivalent to Gromov-Hausdorff distance on any GH precompact family of compact metric spaces when $p = \infty$, but does not define a pseudometric for $p < \infty$. 
\end{Theorem}

The theorem is proved in the Appendix. It is based on the observation that any semi-correspondence can be approximated by a measurable semi-correspondence with an arbitrarily small change in distortion.

\begin{Remark}
    The version of generalized MDS used in~\cite{bronstein2006generalized} that we mentioned in the related work section is the asymmetric modified GH problem, $\inf_{f:X \to Y} \mathrm{dis}(f)$ and its $\ell^p$ relaxation, which the authors of that paper apply to find embeddings of subsets of geodesic spaces to aid in partial surface matching. Our results Corollary~\ref{cor:connection_to_MDS} and Theorem~\ref{thm:srgh_equals_srgw} give a cohesive connection between various ideas: the embedding problem considered in~\cite{bronstein2006generalized} is an application of the modified GH distance of~\cite{memoli2012some}, which is equal to the $p=\infty$ version of a semi-relaxed GW distance, whereas the $p=2$ version of srGW first introduced in~\cite{vincent2022semi} generalizes the standard MDS problem.
\end{Remark}

\section{Numerical Implementation and Experiments}\label{sec:experiments}

\paragraph{Implementation.} In this section we describe how to use the semi-relaxed Gromov Wasserstein distance to find an embedding of a finite metric space $(X,d)$ into a smooth Riemannian manifold $Y$. The algorithm is straightforward: in short, we solve a srGW problem to embed $X$ into a predefined finite subset of $Y$, then use this as initialization for gradient descent of the MDS functional~\eqref{eqn:MDS}, with $Y$ as the target metric space, which could be more general than $\R^m$. We now provide some details.

Given a finite metric space $(X,d)$, we construct an embedding $\hat{f}:X \to Y$ as follows. We first pick a discrete finite subset $S \subseteq Y$ to map into. In practice, we often use a grid in some coordinate system for $Y$, and perturb the points slightly to make it easier to solve the resulting optimization problem. We then solve the semi-relaxed Gromov-Wasserstein problem \eqref{eqn:srGW}, $d_{\mathrm{srGW},2}(X,S)$, which, by Theorem~\ref{thm:monge}, yields an optimal Monge map $f:X \to S$. The computation of $d_{\mathrm{srGW},2}(X,S)$ is approximated via the srGW implementation in the Python Optimal Transport package~\cite{flamary2021pot}; although this approximation is not guaranteed to yield a Monge map, we found that it does so in practice.

To compute the required embedding $\hat{f}: X \to Y$, we run a gradient descent, initialized with the embedding $f$ (sometimes with a small perturbation); as we will see below, this drastically improves the likelihood of finding a good local minimum by gradient descent. Let $X = \{x_1, \ldots, x_n\}$ and initialize a set of $n$ points $(y_i)_{1 \leq i \leq n}$ in $Y$ via $y_i = f(x_i)$. We then consider the distortion function $Y^n \to \R$ defined by
\begin{equation}\label{eqn:discretedis}
    (y_1,\ldots,y_n) \mapsto  \frac{1}{2}\sum_{i,j=1}^n \left( d_X(x_i, x_j) - d_Y(y_i, y_j)\right)^2
\end{equation}
(cf.\ the MDS functional \eqref{eqn:MDS}) and use a gradient-based method on $Y$ to find a local minimum $(\hat{y}_1, \hat{y}_2,\ldots,\hat{y}_n)$. Our embedding is then given by $\hat{f}(x_i) = \hat{y}_i$. We refer to this method as \define{SRGW+GD}. In our examples below, we use the Adam optimizer~\cite{kingma2014adam}.

In practice, the desired embedding space $Y$ may only be known up to certain hyperparamters, such as scale. Our main example below will be when $Y$ is a circle of unknown radius. If $Y$ depends on a scale factor (or multiple scale factors), we add the scale factor as an additional variable in \eqref{eqn:discretedis}.

\paragraph{MNIST.} We first benchmark SRGW+GD against other dimension reduction methods in the case when the target space is Euclidean. We use the MNIST dataset consisting of 70,000 $28\times 28$ grayscale images of handwritten digits, which we view as vectors in $\mathbb{R}^{784}$, separated into ten smaller datasets (MNIST0 to MNIST9), one for each digit. We embed each dataset into $\mathbb{R}^2$ using PCA, SMACOF MDS, and SRGW+GD. Table \ref{tab:euclidean} reports the distortion ($\dis_2$) values for each embedding. We see that SRGW+DG and SMACOF achieve similar distortion in all cases, with SRGW+GD performing slightly better. In terms of compute time, we find that SRGW+GD is between 3 and 14 times faster then SMACOF MDS (using scikit-learn). 

\begin{table}[ht]
    \centering
    \begin{tabular}{l|c|c|c|c|}
    & PCA & SMACOF MDS & SRGW+GD \\ \hline
    MNIST0 & 2.580 & 1.556 &  1.554 \\ \hline
    MNIST1 & 1.290 & 0.794 & 0.777 \\ \hline
    MNIST2 & 3.070 & 1.761 & 1.751 \\ \hline
    MNIST3 & 2.792 &  1.594 & 1.586 \\ \hline
    MNIST4 & 2.668 & 1.523 & 1.518 \\ \hline
    MNIST5 & 2.702 & 1.602 & 1.594 \\ \hline
    MNIST6 & 2.549 &  1.495 & 1.477 \\ \hline
    MNIST7 & 2.346 & 1.350 & 1.342\\ \hline
    MNIST8 & 2.901 & 1.646 & 1.623 \\ \hline
    MNIST9 & 2.421 & 1.405 & 1.388 \\ \hline
    \end{tabular}
    \caption{ Distortion for embeddings of MNIST data into $\mathbb{R}^2$.}\label{tab:euclidean}
\end{table}

\paragraph{Rotated MNIST.}

In order to demonstrate the performance of SRGW+GD when the target space is non-Euclidean, we artifically introduce non-linear structure into the MNIST datasets. We apply a random rotation between $0^{\circ}$ and $360^{\circ}$ to each image in each MNIST dataset (filling gaps with black pixels) to create new datasets R-MNIST0 to R-MNIST9. We embed these datasets into $\mathbb{R}^2$ using t-SNE, PCA, SMACOF MDS and SRGW+GD. We also embed each dataset into a circle of unknown radius using SRGW+GD and the following comparison methods:
\begin{itemize}[leftmargin=*]
    \item \define{GD} denotes minimizing (\ref{eqn:discretedis}) with a random initalization and the Adam optimizer. We use 10 random initializations and report the min and max distortion over all trials.
    \item \define{CC} is the circular coordinate method introduced in~\cite{de2009persistent}, which constructs a map from a finite metric space $X$ to $S^1$ using persistent cohomology. We used the density-robust version introduced in~\cite{paik2023circular}. Since the method only produces circular coordinates and not a radius, we estimate the radius as $\max_{x,y \in X} d(x,y)/\pi$. 
    \end{itemize}

Distortion values are contained in Table~\ref{tab:distortiontable}.\footnote{While t-SNE does not aim to reduce distortion and thus is not a fair comparison, it is a helpful contrast for the qualitative behavior of SRGW+GD; we include distortion values for completeness} For embeddings into $\mathbb{R}^2$, we again find that SRGW+GD and SMACOF MDS have similar distortion values, all significantly lower than PCA. SRGW+GD achieves higher distortion on $S^1$ than on $\mathbb{R}^2$ (since the target space has one less dimension and thus captures less variation), but still achieves a lower distortion on $S^1$ than PCA achieves on $\mathbb{R}^2$. SRGW+GD achieves a significantly lower distortion on $S^1$ than CC does, and lower distortion than GD  in all cases except for some trials on R-MNIST1.

We can also gain some insight into how SRGW+GD differs qualitatively from other methods. In Figure \ref{fig:scatterplots}, we show the image of the embedding for the dataset R-MNIST9, colored by the true angle of rotation. We also plot the true angle against the inferred angular coordinate (for $\mathbb{R}^2$ this is taken to be the angle of the point from the $x$-axis). We note that CC maps true rotation angles to angular coordinates in a roughly injective way, thereby capturing the rotation process accurately. t-SNE maps the rotation angle roughly injectively onto a thickened curve in the plane, but inferring an angular coordinate with our naive method does not recover the rotation angle correctly. PCA, MDS and both versions of SRGW+GD map true rotation angles to angular coordinates via a roughly degree two map. This better captures the global geometry, since an image of a $9$ is often closer to its $180^{\circ}$ rotation than its $90^{\circ}$ rotation. Finally, we see that GD recovers none of the rotation structure in the dataset.

\begin{table*}[ht]
    \centering
    \begin{tabular}{l|c|c|c|c||c|c|c|}
    & \multicolumn{4}{|c||}{$\mathbb{R}^2$ embeddings} & \multicolumn{3}{|c|}{$S^1$ embeddings} \\ \hline
         & t-SNE & PCA & SMACOF MDS & SRGW+GD & CC   & GD (min,max)  & SRGW+GD  \\ \hline
         R-MNIST0 & 28.196 & 3.508 & 1.989 & 1.987 & 5.467  & (2.810, 2.812) & 2.534 \\ \hline
         R-MNIST1 & 33.631 & 2.032 & 1.200 & 1.202 & 1.818  & (1.709, 2.247) & 1.702 \\ \hline
         R-MNIST2 & 30.199 & 3.615 & 2.020 & 2.019 & 5.586  & (2.854, 2.856) & 2.591 \\ \hline
         R-MNIST3 & 30.410 & 3.261 & 1.873 & 1.873 & 5.459  & (2.787, 2.790) & 2.439 \\ \hline
         R-MNIST4 & 29.236 & 3.593 & 1.880 & 1.828 & 4.846 & (2.488, 2.490) & 2.329 \\ \hline
         R-MNIST5 & 28.459 & 3.159 & 1.812 & 1.812 & 5.325  & (2.734, 2.736) & 2.401 \\ \hline
         R-MNIST6 & 32.086 & 3.598 & 1.942 & 1.900 & 2.758  & (2.621, 2.623) & 2.437 \\ \hline
         R-MNIST7 & 32.798 & 3.347 & 1.828 & 1.796 & 2.594  & (2.537, 2.539) & 2.324 \\ \hline
         R-MNIST8 & 28.310 & 3.296 &  1.870 & 1.869 & 5.368 & (2.735, 2.737) & 2.418  \\ \hline
         R-MNIST9 & 32.468 & 3.307 & 1.783 & 1.775 & 2.638 & (2.464, 2.466) & 2.280 \\ \hline
    \end{tabular}
    \caption{Distortion ($\dis_2$) for embeddings of randomly rotated MNIST data using various methods. }
    \label{tab:distortiontable}
\end{table*}

\begin{figure*}
\centering
\resizebox{0.95\textwidth}{!}{
    \begin{tikzpicture}[scale=1.0]
        \node[rotate=90] at (-1.1,-2.5) {\tiny angular coord.};
        \node at (0.1,-3.6) {\tiny rotation angle};
        \node at (13.7,-1.5) {\includegraphics[width=1.3cm]{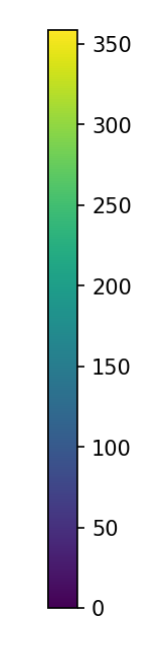}};
        
        \node at (0,1) {t-SNE};
        \node at (0,-0.3) 
        {\includegraphics[width=2cm]{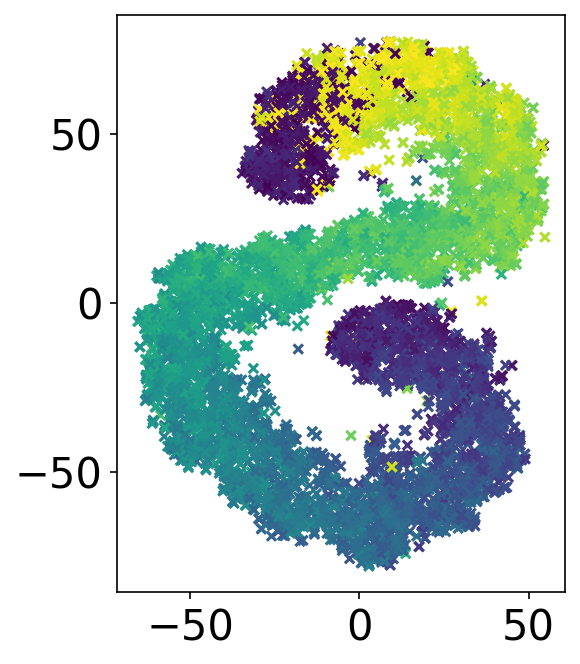}};
        \node at (0,-2.5) 
        {\includegraphics[height=2cm]{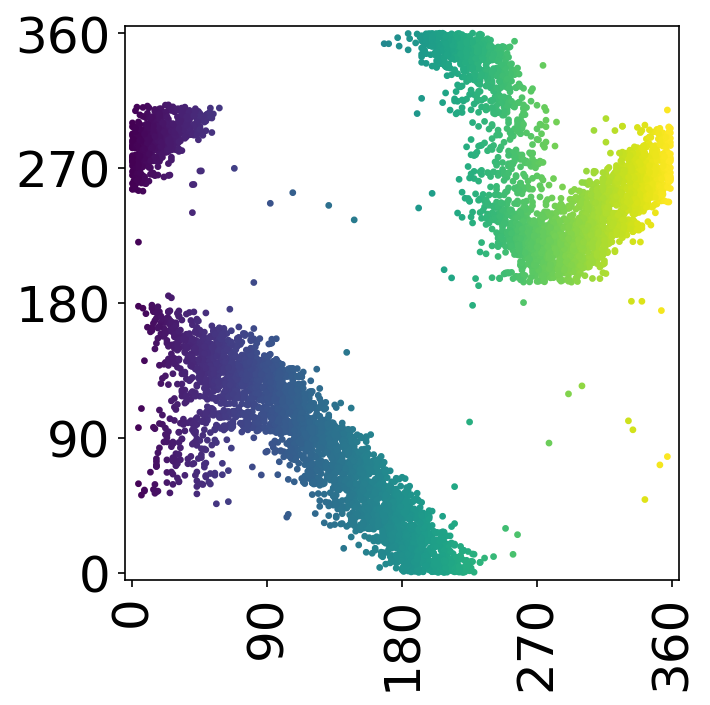}};
        
        \node at (2,1) {PCA};
        \node at (2,-0.3) 
        {\includegraphics[width=2cm]{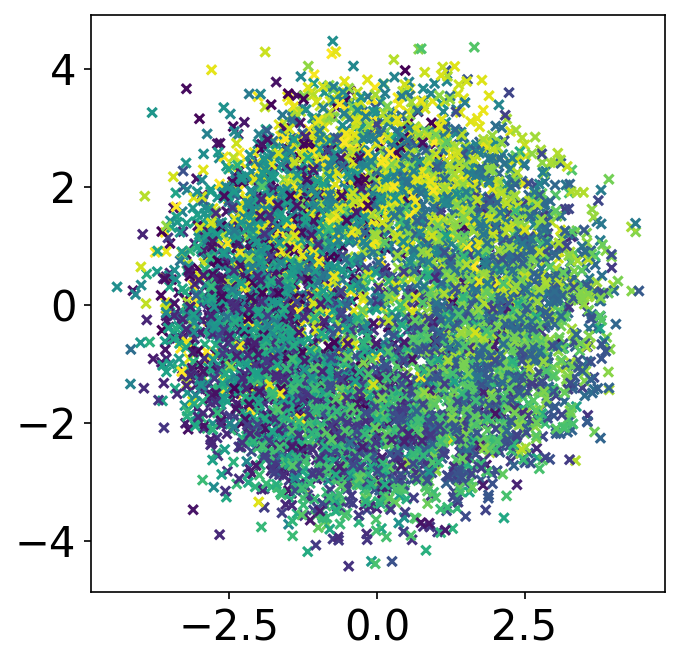}};
        \node at (2,-2.5) {\includegraphics[height=2cm]{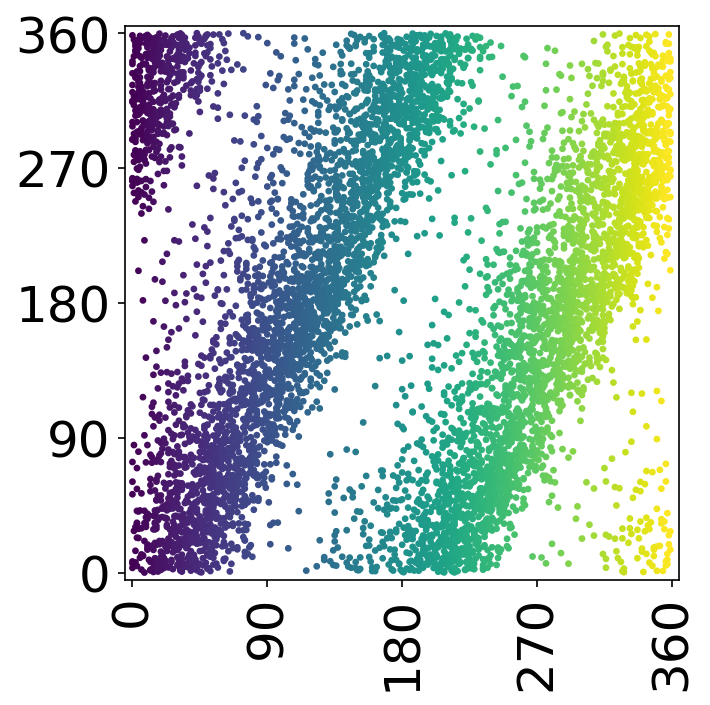}};
        
        \node at (4,1) {MDS};
        \node at (4,-0.3) 
        {\includegraphics[width=2cm]{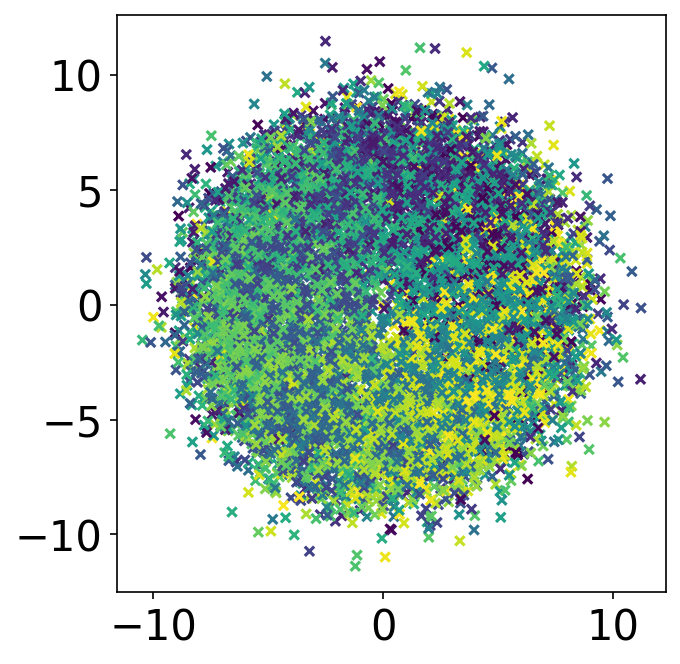}};
        \node at (4,-2.5) {\includegraphics[height=2cm]{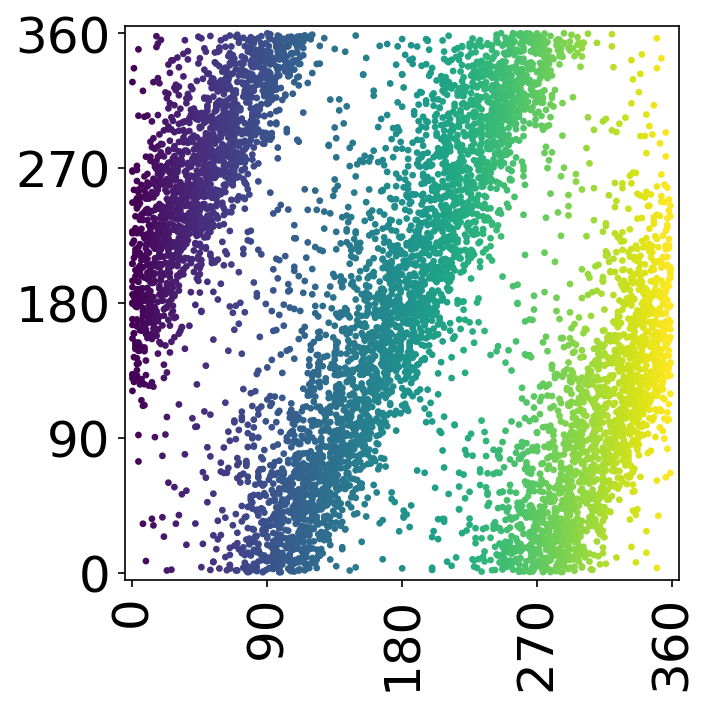}};

        \node at (6,1) {SRGW+GD};
        \node at (6,-0.3) 
        {\includegraphics[width=2cm]{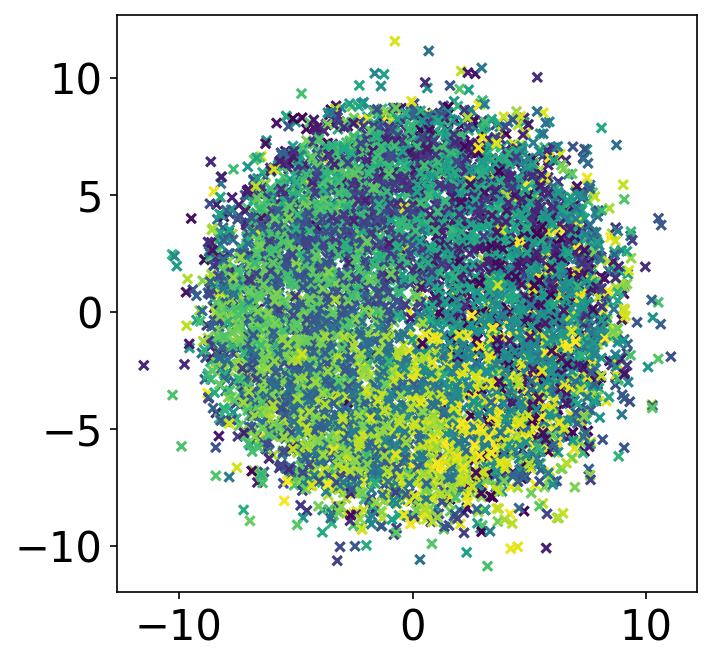}};
        \node at (6,-2.5) {\includegraphics[height=2cm]{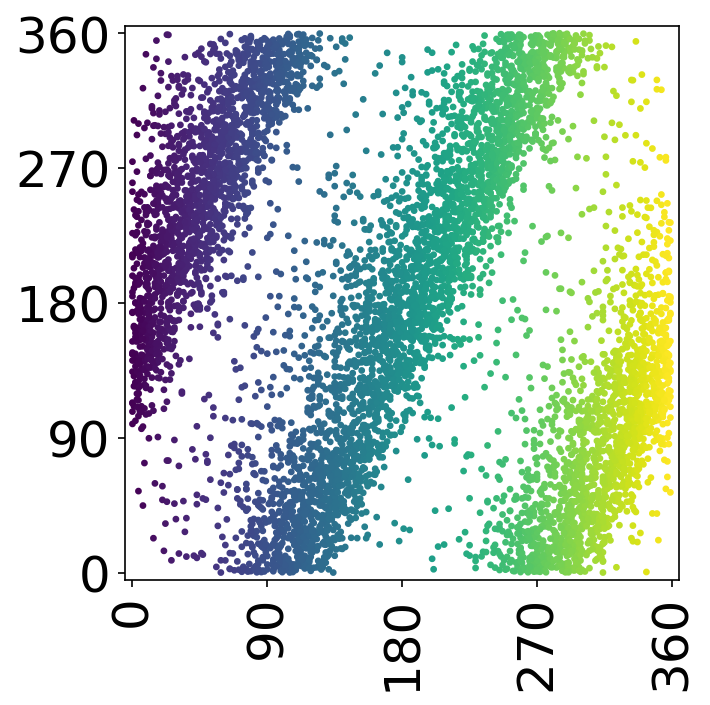}};
        
        \node at (8,1) {CC};
        \node at (8,-0.3) 
        {\includegraphics[width=2cm]{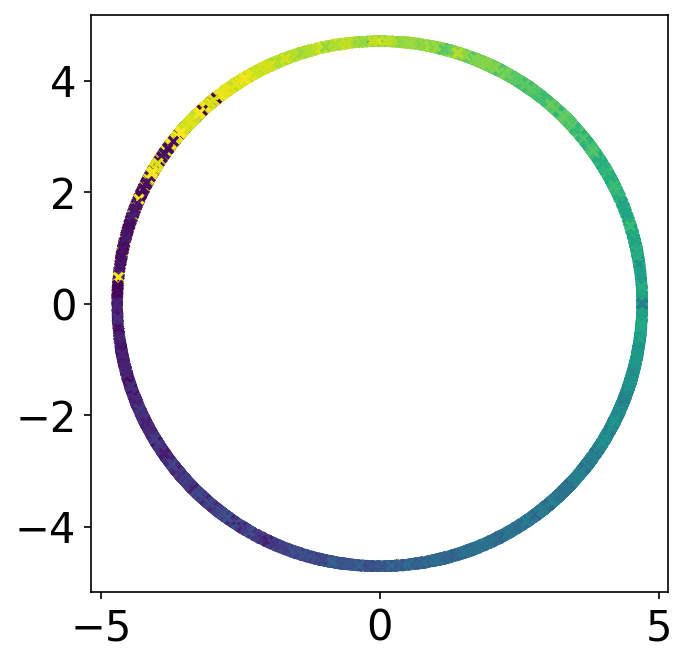}};
        \node at (8,-2.5) {\includegraphics[height=2cm]{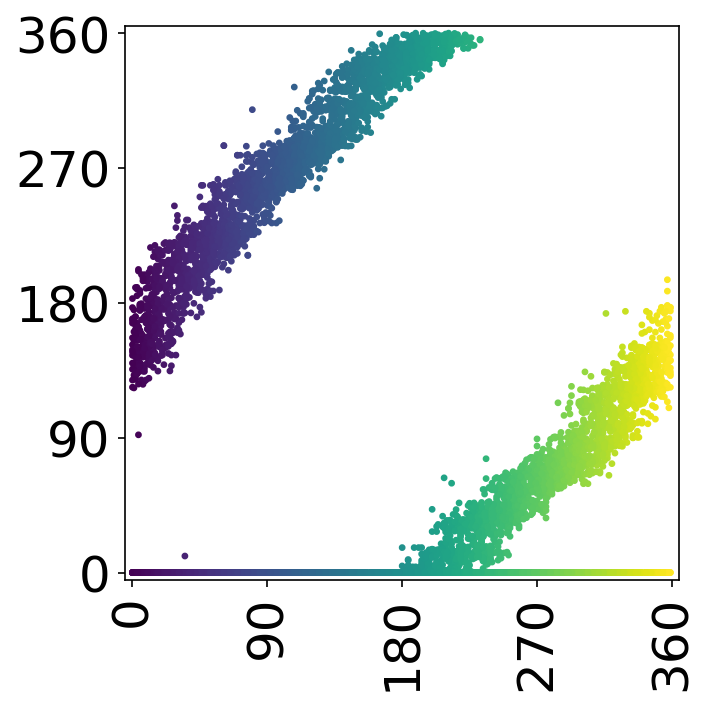}};
        
        \node at (10,1) {GD};
        \node at (10,-0.3) 
        {\includegraphics[width=2cm]{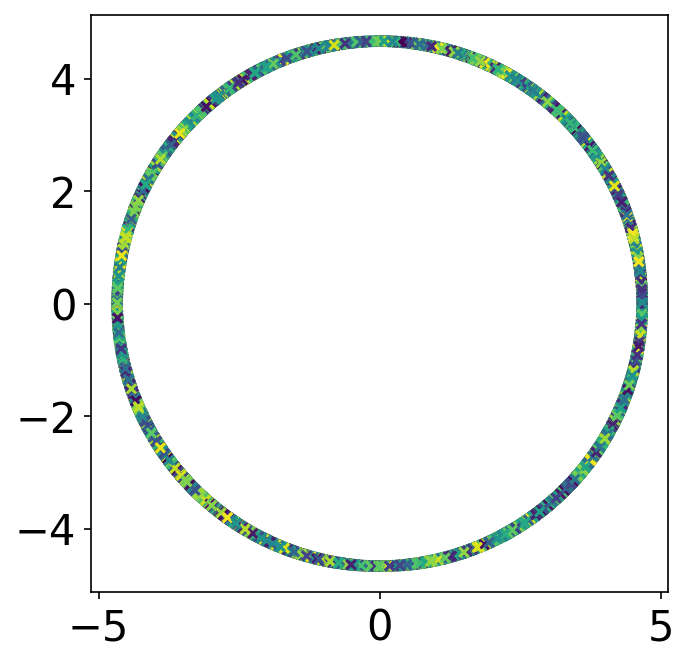}};
        \node at (10,-2.5) {\includegraphics[height=2cm]{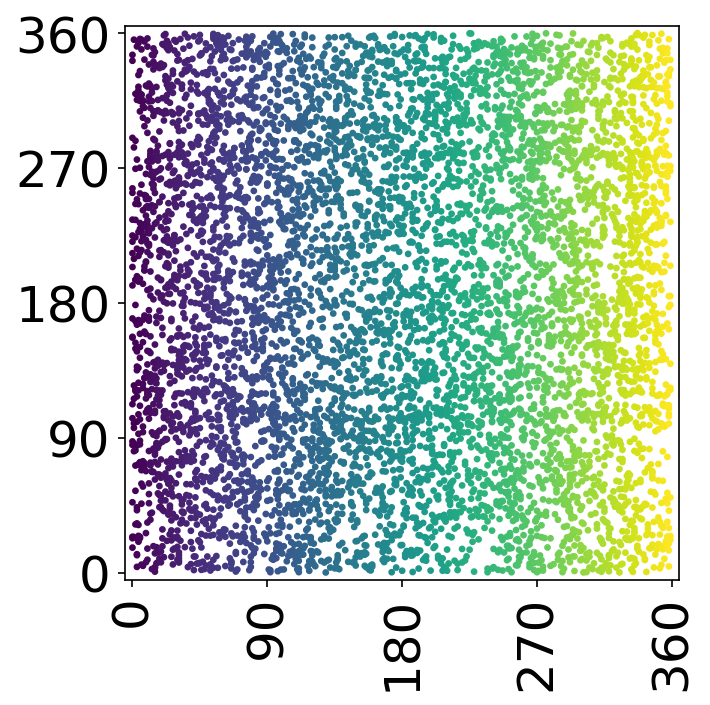}};

        \node at (12,1) {SRGW+GD};
        \node at (12,-0.3) 
        {\includegraphics[width=2cm]{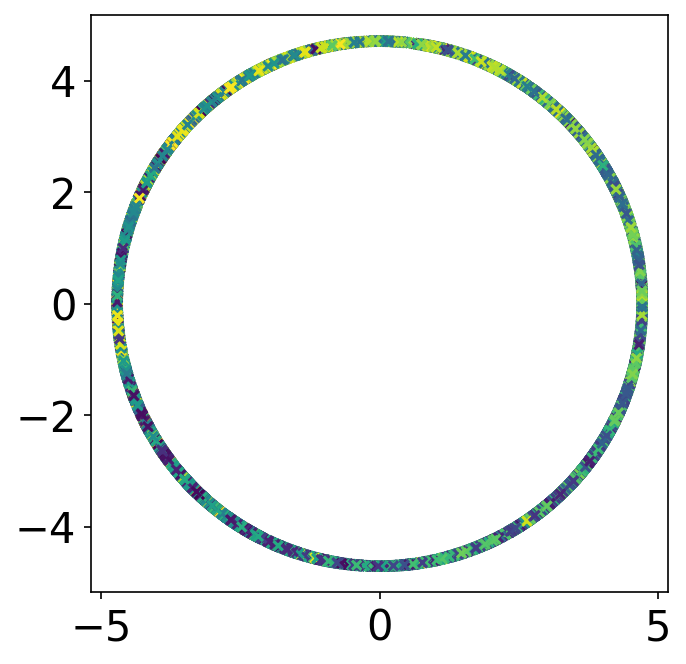}};
        \node at (12,-2.5) {\includegraphics[height=2cm]{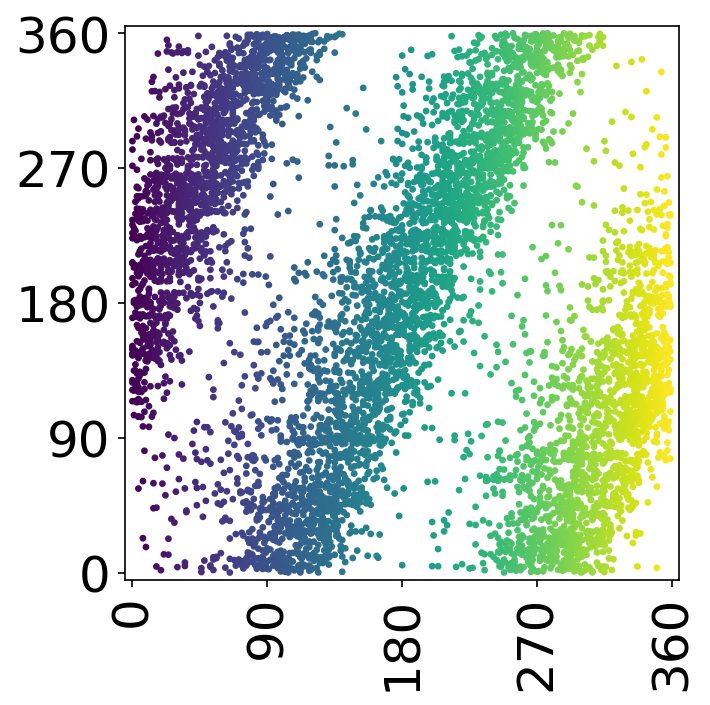}};
    \end{tikzpicture}
    }
    \caption{Circle and planar embeddings for the R-MNIST9 dataset (above), and plots comparing the true angle of rotation vs the inferred angular coordinate from the embedding (below). Color indicates the true angle of rotation.}
    \label{fig:scatterplots}
\end{figure*}

In general, it can be hard to infer an angular coordinate from a planar embedding (see t-SNE in Figure \ref{fig:scatterplots}). Even our naive method above requires finding an appropriate center for the data, which might not be the mean if the data is distributed very unevenly (e.g. in the redistricting application below). The advantage of choosing a circle as the target space is that it produces a well-defined angular coordinate. Our experiments demonstrate that regardless of which target space is preferred, SRGW+GD effectively preserves global geometry. They also demonstrate the necessity of srGW embeddings as an initialization point for gradient descent.

\paragraph{Cities.}

To demonstrate a non-Euclidean embedding where approximate isometric embedding is possible, we use a list of the 20 largest cities\footnote{\url{https://simplemaps.com/data/world-cities}, CC-BY 4.0 license}, with the geodesic distance on the Earth between every pair of cities as ground truth (this distance does not assume the Earth is a perfect sphere, and instead uses the WGS-84 ellipsoid). Using SRGW+GD, we embed this dataset into a sphere of radius 6371 (the average radius of the Earth in kilometers). Figure \ref{fig:earth} shows the embedding. SRGW+GD achieves an embedding that is almost isometric; the pairwise distances between the embedded points never differ by more than $14$ kilometers from the true distance. By contrast, an MDS embedding into $\mathbb{R}^3$ achieves a distortion of $264.724$, indicating the benefit of choosing a target space with the appropriate metric (a geodesic sphere).

\begin{figure}[ht]
    \centering
    \includegraphics[width=0.24\textwidth]{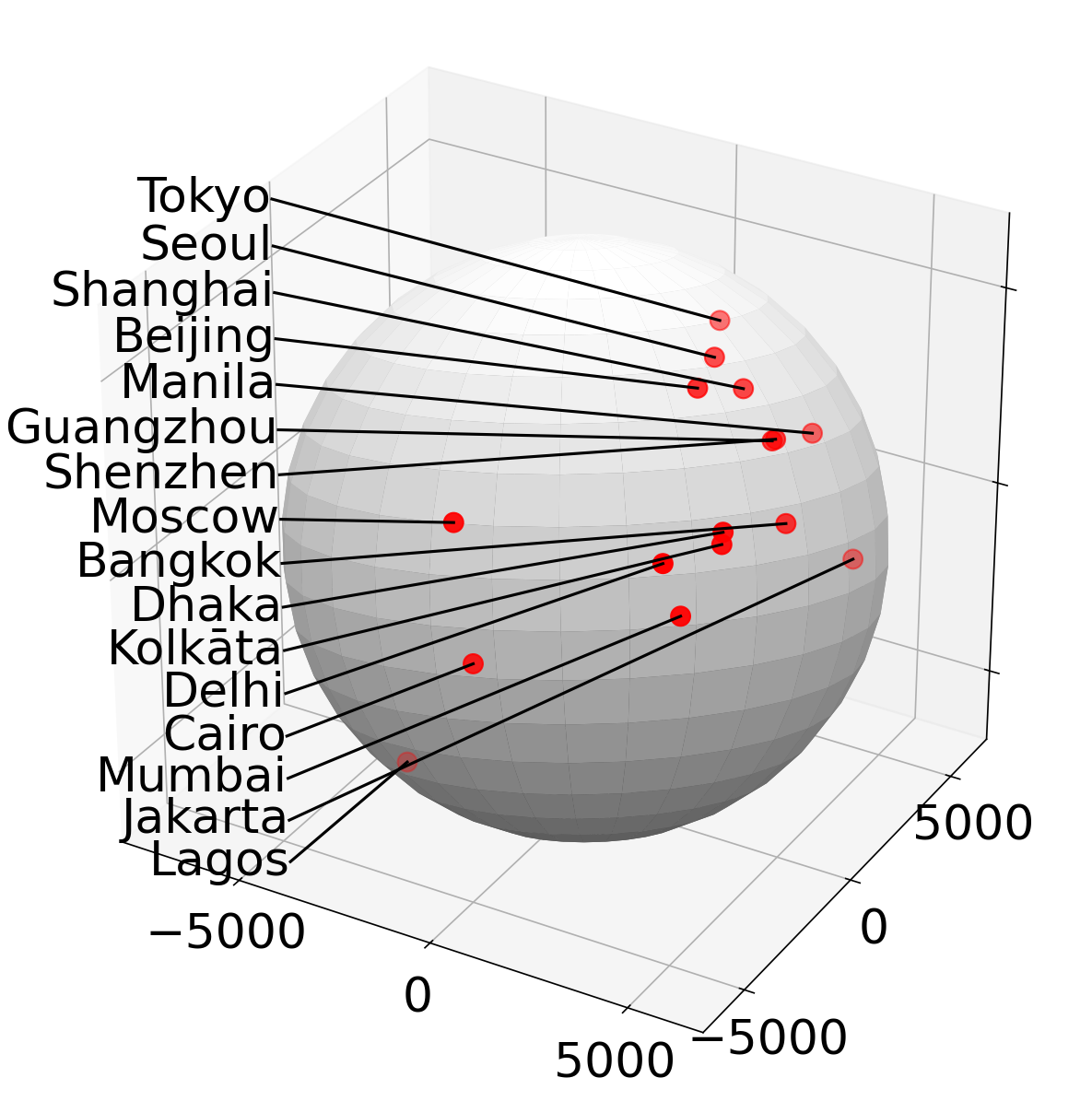}%
    \includegraphics[width=0.24\textwidth]{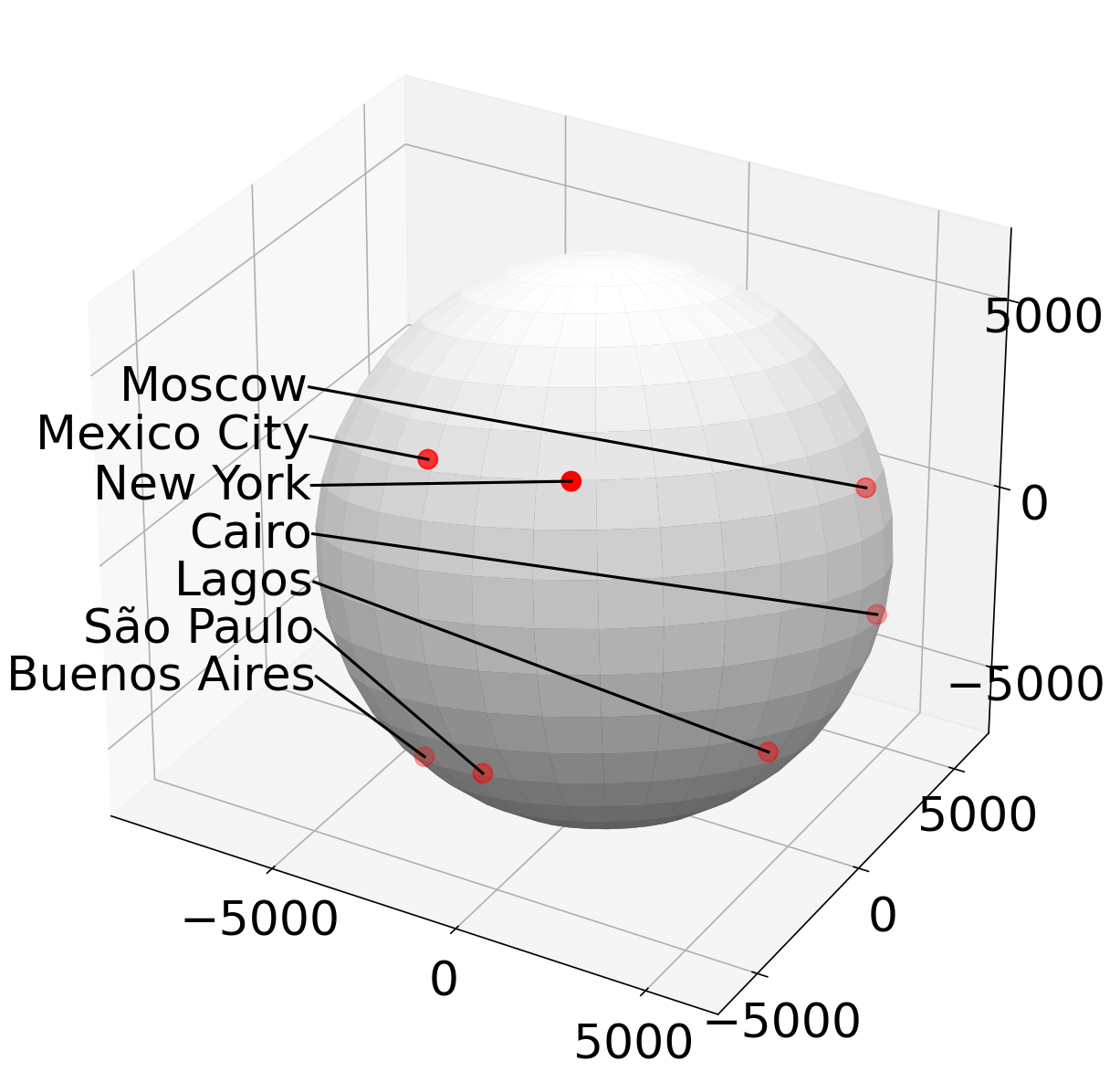}
    \caption{Embedding onto a geodesic sphere of 20 world cities.}
    \label{fig:earth}
\end{figure}

\section{Application to Redistricting}\label{sec:redistricting}

We now demonstrate how an embedding into a natural non-Euclidean target can enable visualization of a complex data set, resulting in important insights, using computational redistricting as our area of application.

\paragraph{Background and Data.} Redistricting is the process of dividing a region into contiguous, equal population districts for the purposes of electing representatives. There has been a lot of recent attention on generating redistricting \define{ensembles} -- large samples from the space of valid redistricting plans for a given U.S.~state~\cite{chen2015cutting, chikina2017assessing, herschlag2020quantifying, deford2021recombination,duchin2022homological}. When analyzed, ensembles can uncover baseline expectations for a typical plan, or be used to flag outliers (some of which might be so-called gerrymanders, i.e.~unfair maps). We will use our SRGW+GD method to visualize ensembles of two-district plans in order to achieve both these goals, similar to the approach  in~\cite{abrishami2020geometry}.

There are currently six states in the contiguous United States with two Congressional districts: Idaho (ID), Maine (ME), Montana (MT), New Hampshire (NH), Rhode Island (RI) and West Virginia (WV). For each of these states we obtained Census blockgroups from \cite{nhgis} and generated an ensemble of 1,000 redistricting plans using the ReCom algorithm~\cite{deford2021recombination}. As our distance between plans we chose a \define{Hamming distance} where the distance between two redistricting plans is defined as the minimum number of Census blockgroups that must be reassigned to change the first plan into the second. Treating the ensemble as a 1,000-point metric space with this distance, we then embed the ensemble into a circle with SRGW+GD. For each embedded ensemble, we plot the image of the embedding as a set of points on the circle, as shown in Figure \ref{fig:circles}. We also display the average division for each part of the circle, and histograms showing the distribution of circular coordinates in each ensemble. In the Appendix, we try other non-linear planar embeddings and find that none of them reveal the circle structure within the data across all states.

\begin{figure}
    \centering
    \begin{subfigure}{0.37\textwidth}
    \includegraphics[width=0.38\textwidth]{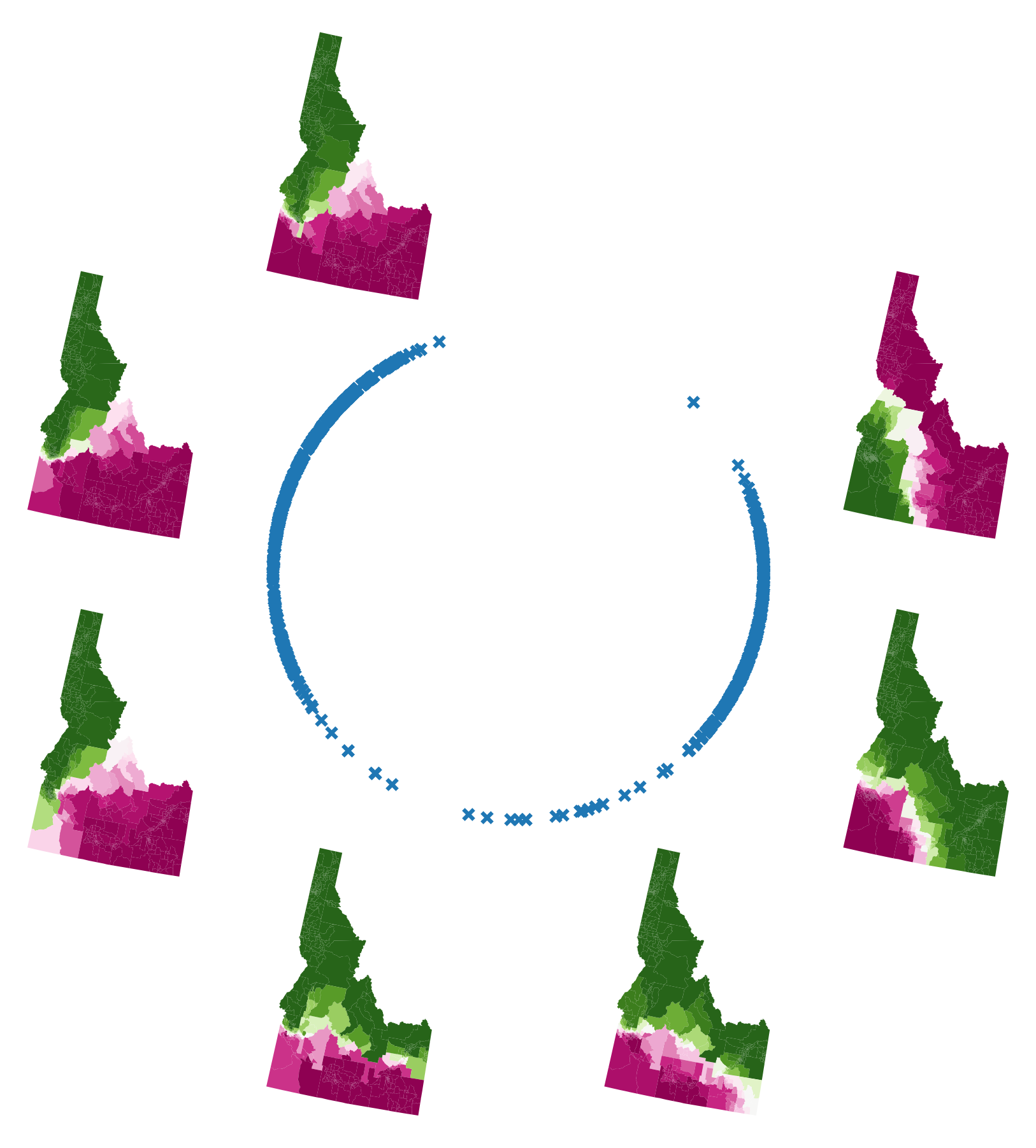}%
    \includegraphics[width=0.6\textwidth]{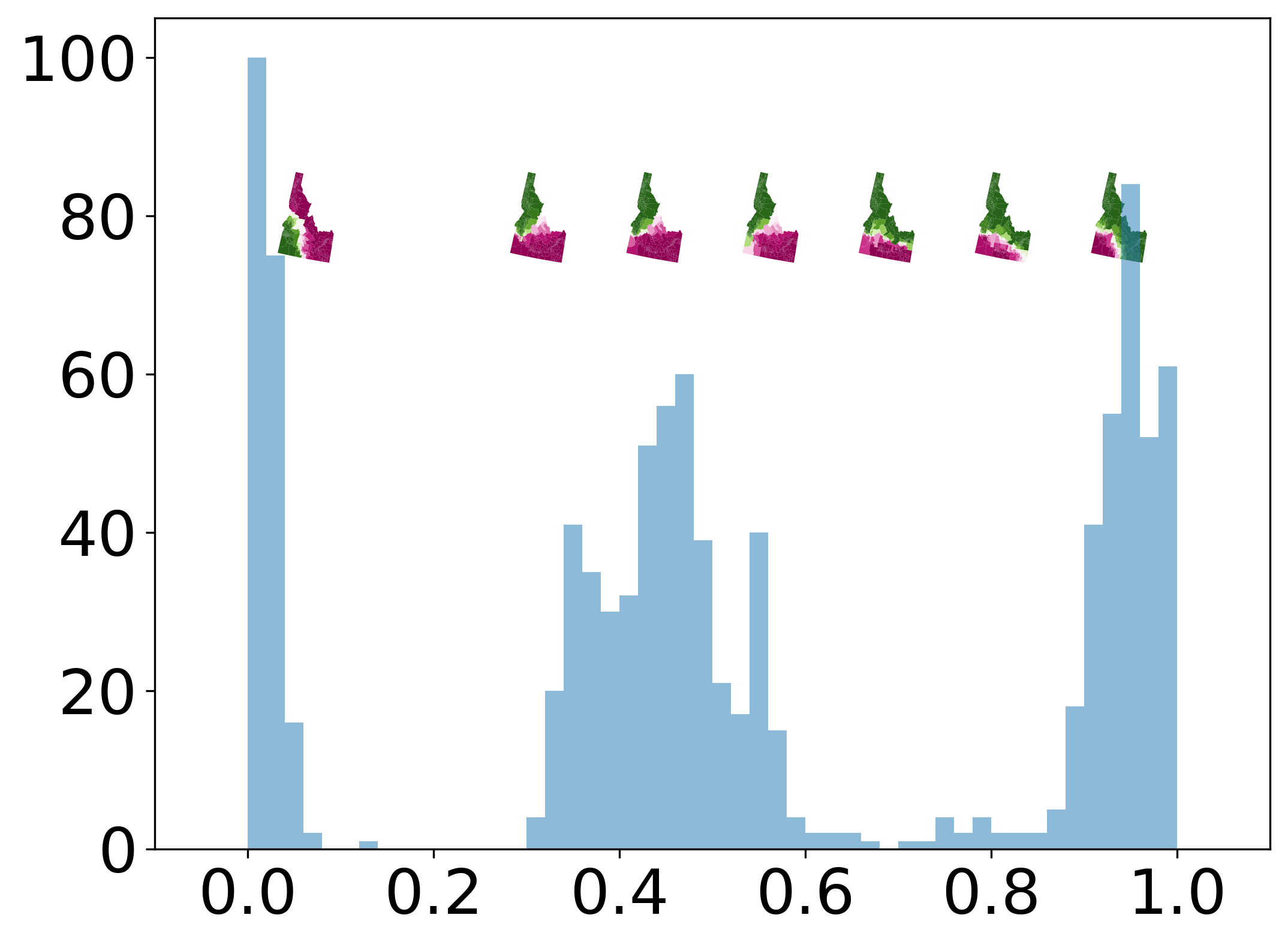}%
    \caption{Idaho (ID)}
    \end{subfigure} %
    \begin{subfigure}{0.37\textwidth}
    \includegraphics[width=0.38\textwidth]{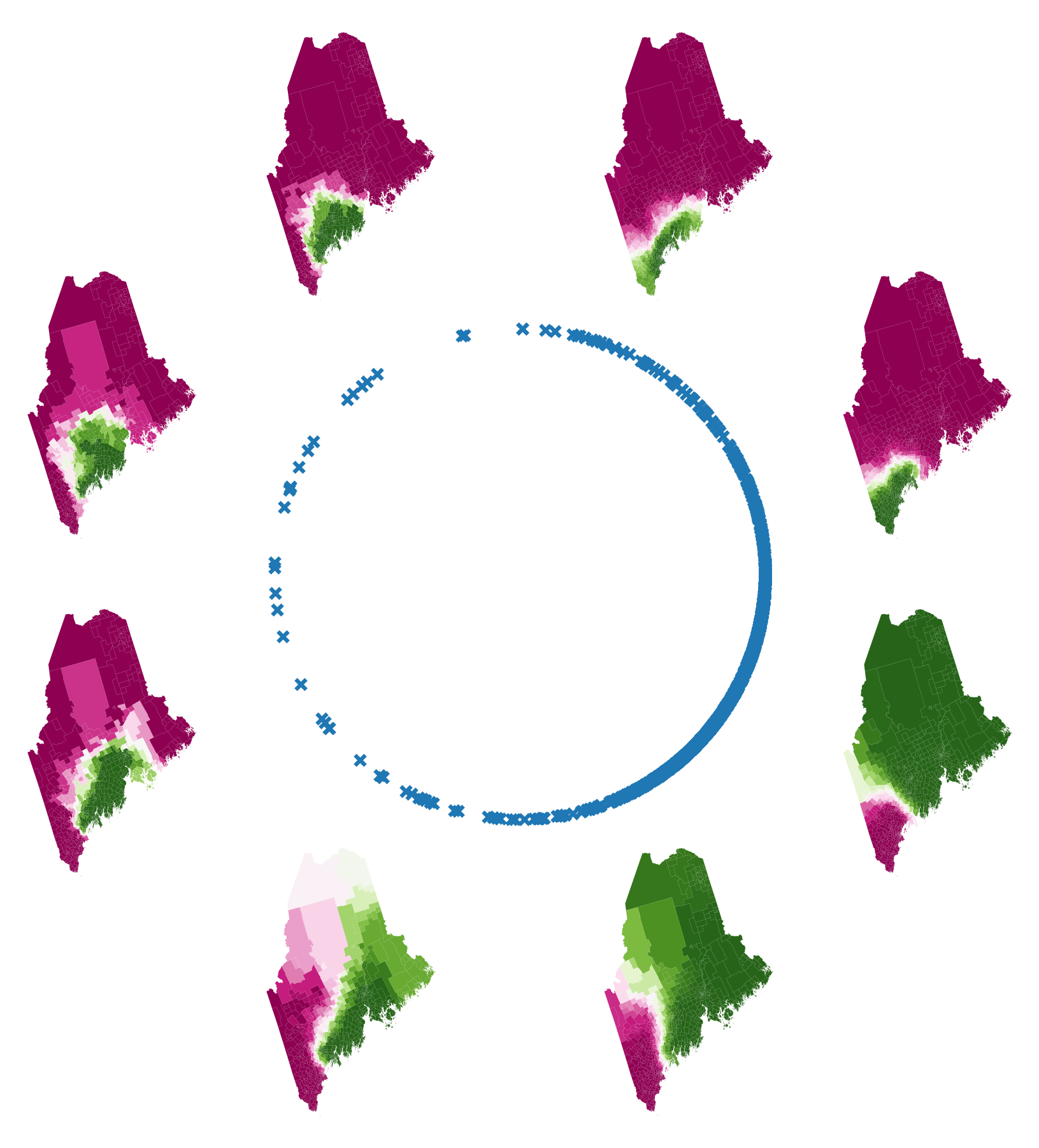}%
    \includegraphics[width=0.6\textwidth]{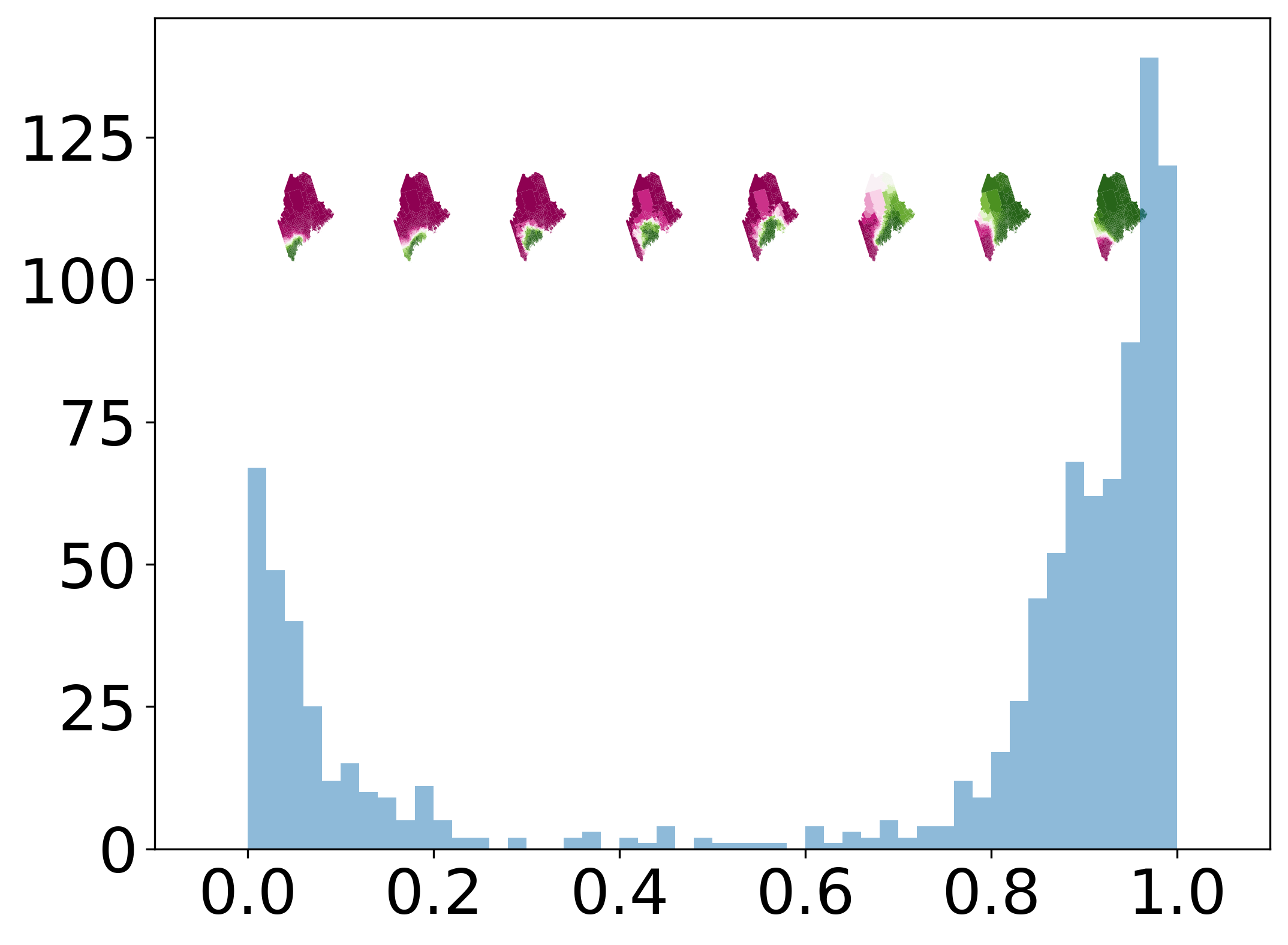}%
    \caption{Maine (ME)}
    \end{subfigure}   
    
    \begin{subfigure}{0.37\textwidth}
    \includegraphics[width=0.38\textwidth]{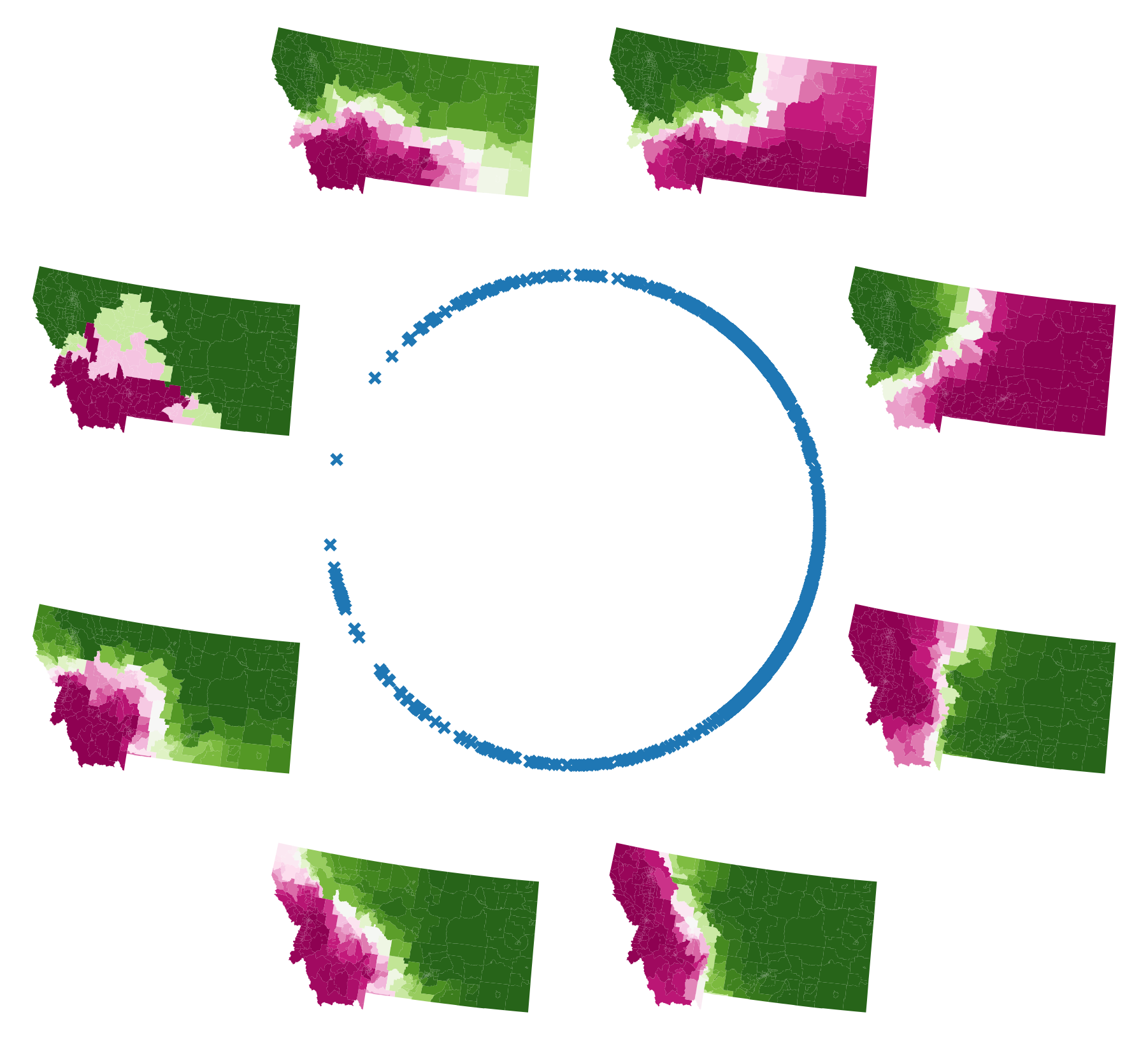}%
    \includegraphics[width=0.6\textwidth]{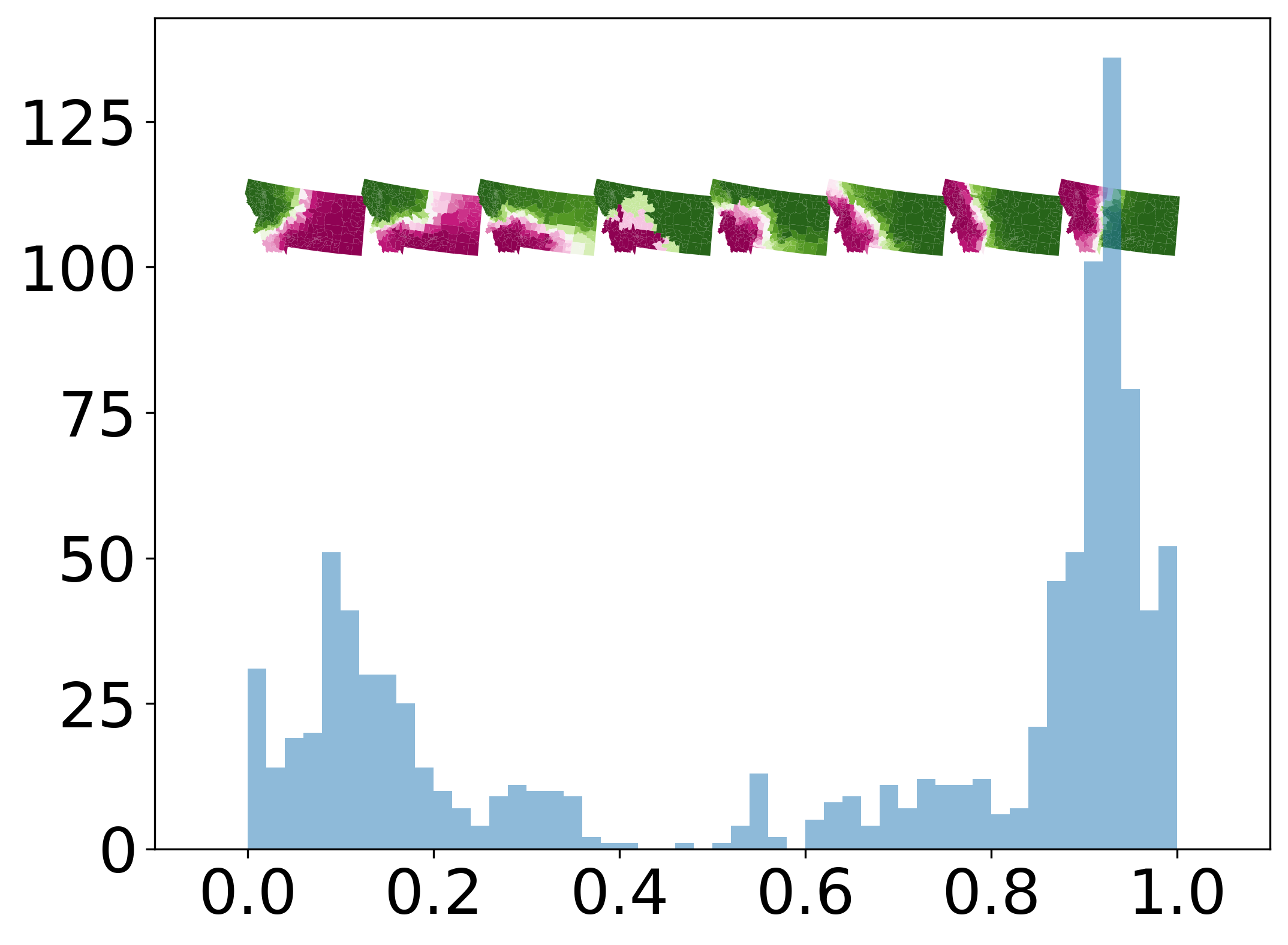}%
    \caption{Montana (MT)}
    \end{subfigure} %
    \begin{subfigure}{0.37\textwidth}
    \includegraphics[width=0.38\textwidth]{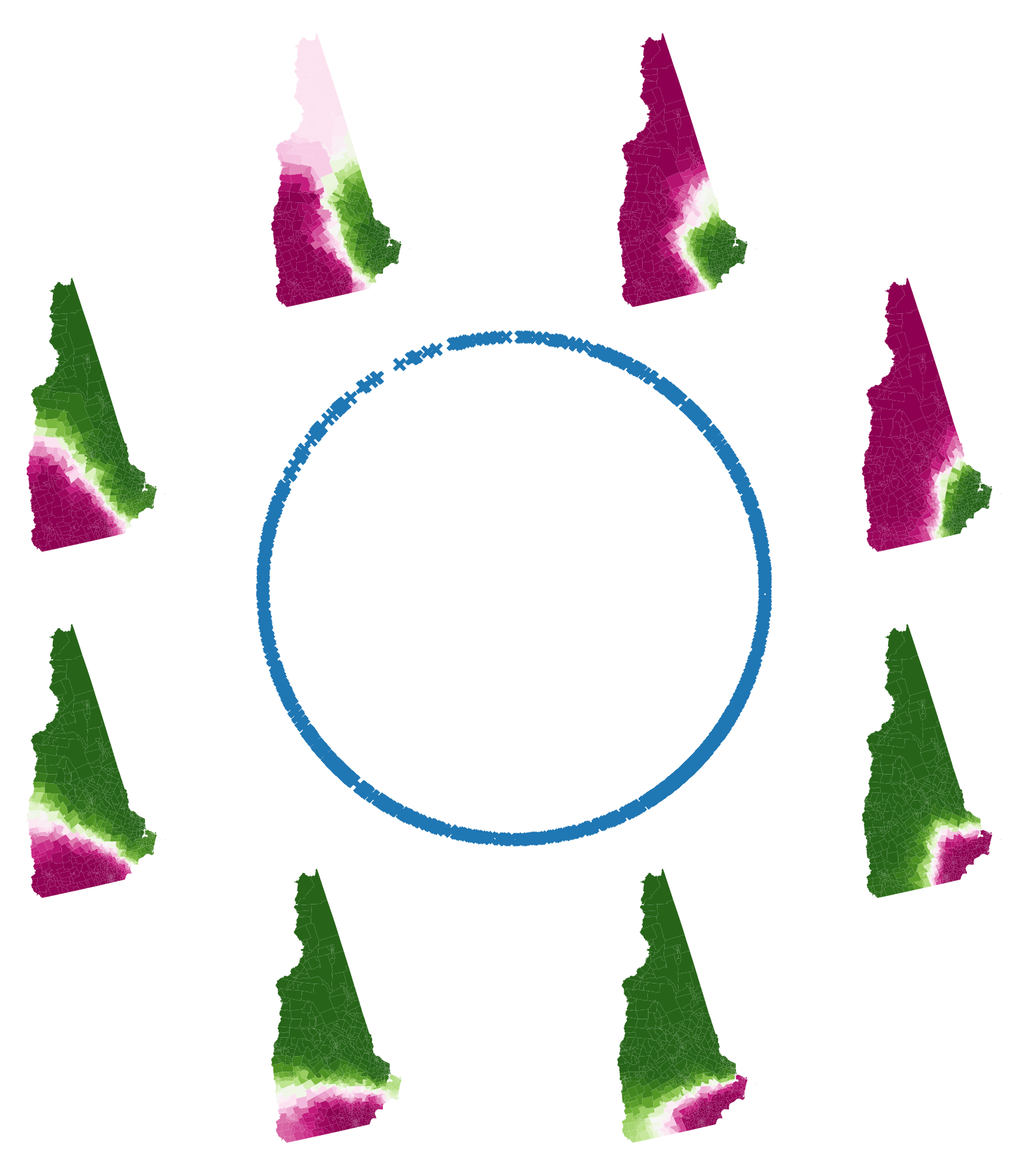}%
    \includegraphics[width=0.6\textwidth]{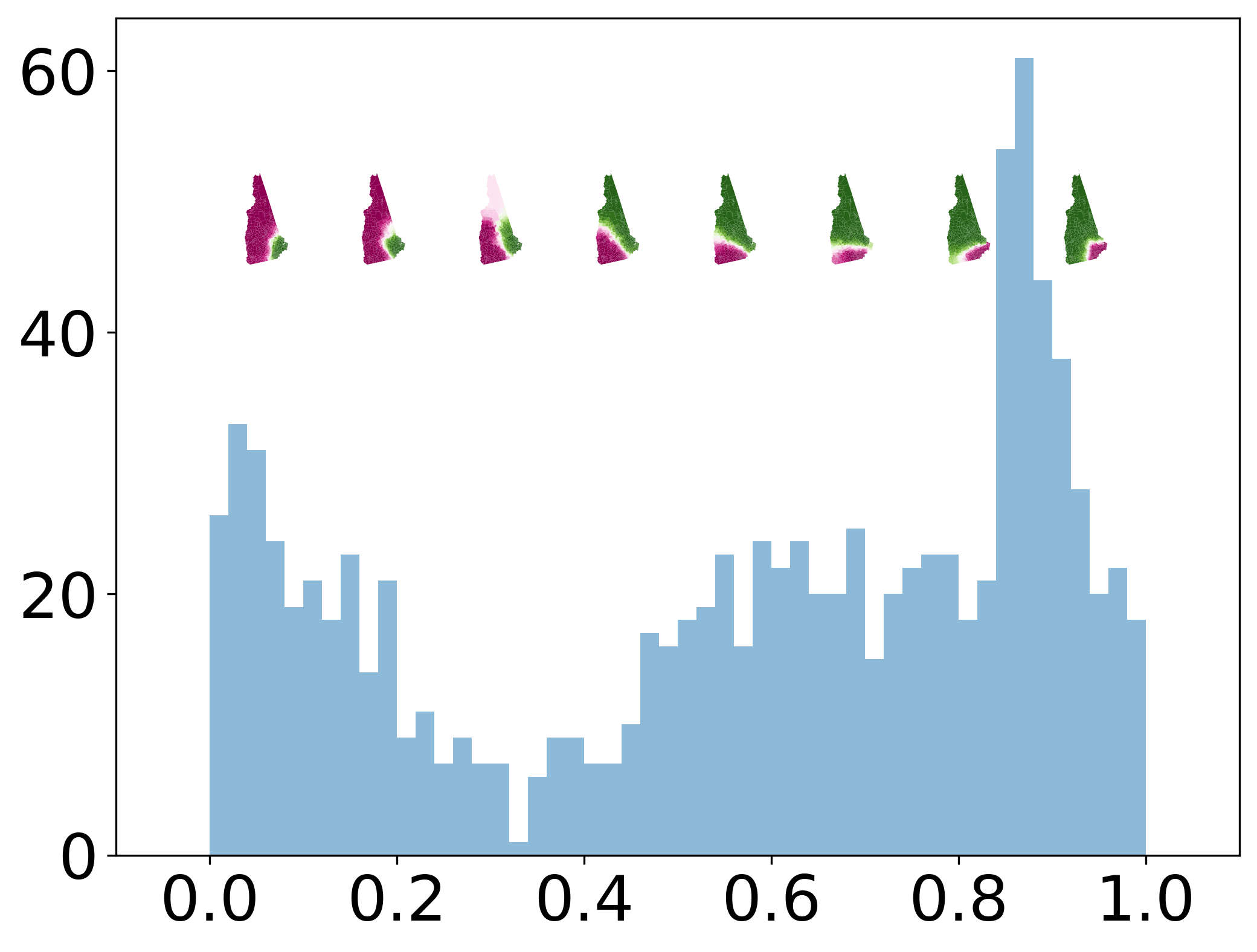}%
    \caption{New Hampshire (NH)}
    \end{subfigure}     
    
    \begin{subfigure}{0.37\textwidth}
    \includegraphics[width=0.38\textwidth]{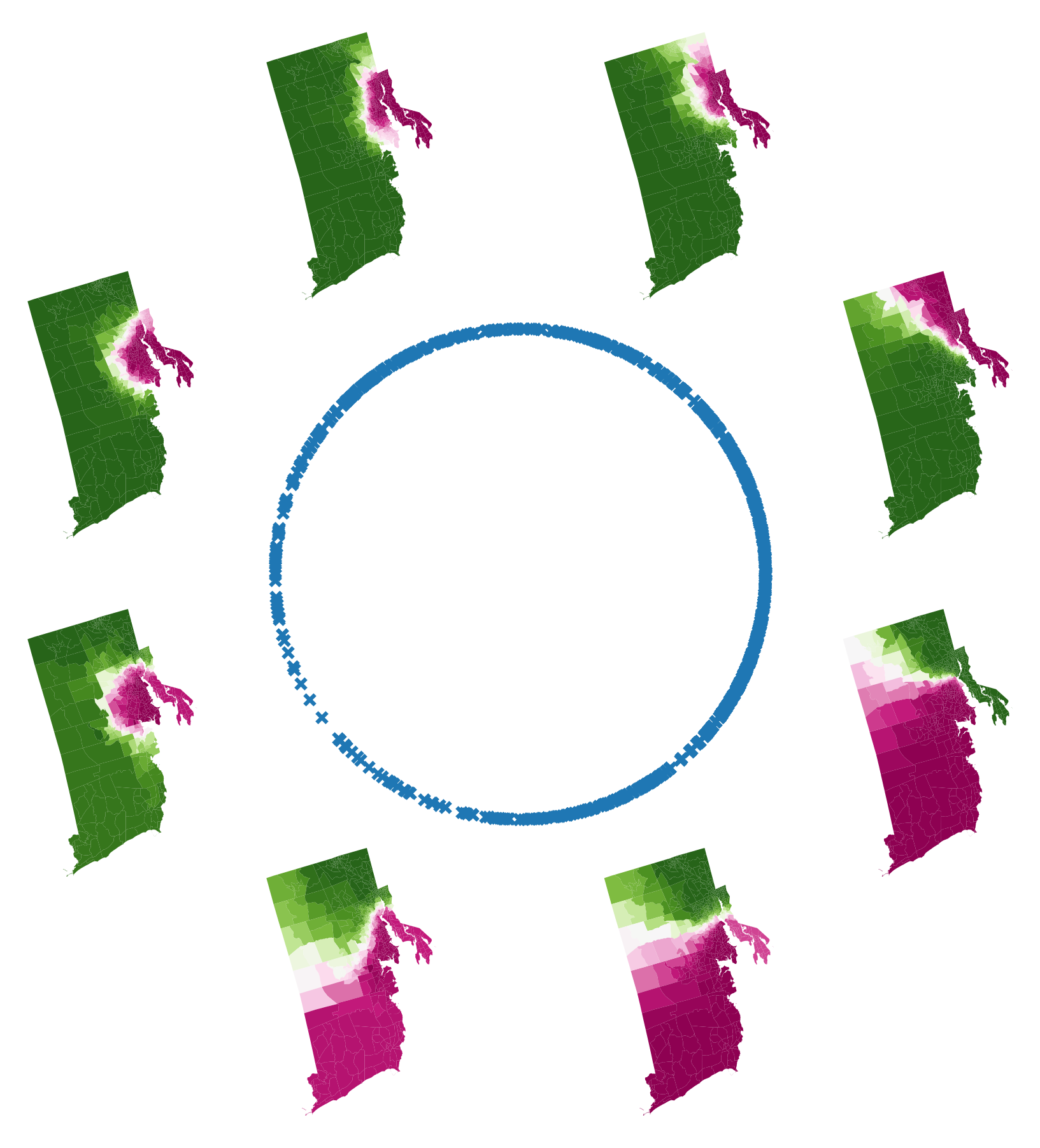}%
    \includegraphics[width=0.6\textwidth]{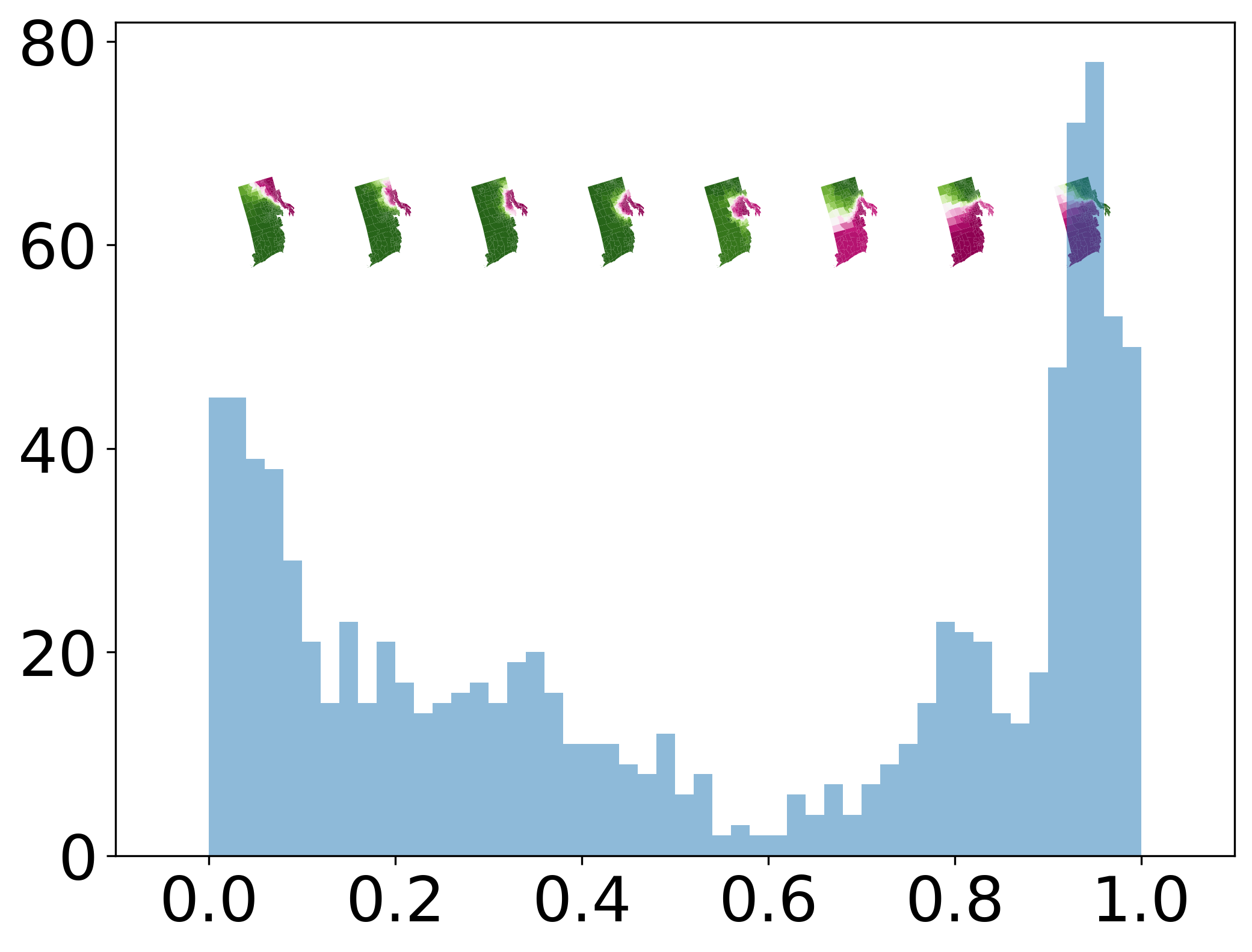}%
    \caption{Rhode Island (RI)}
    \end{subfigure} %
    \begin{subfigure}{0.37\textwidth}
    \includegraphics[width=0.38\textwidth]{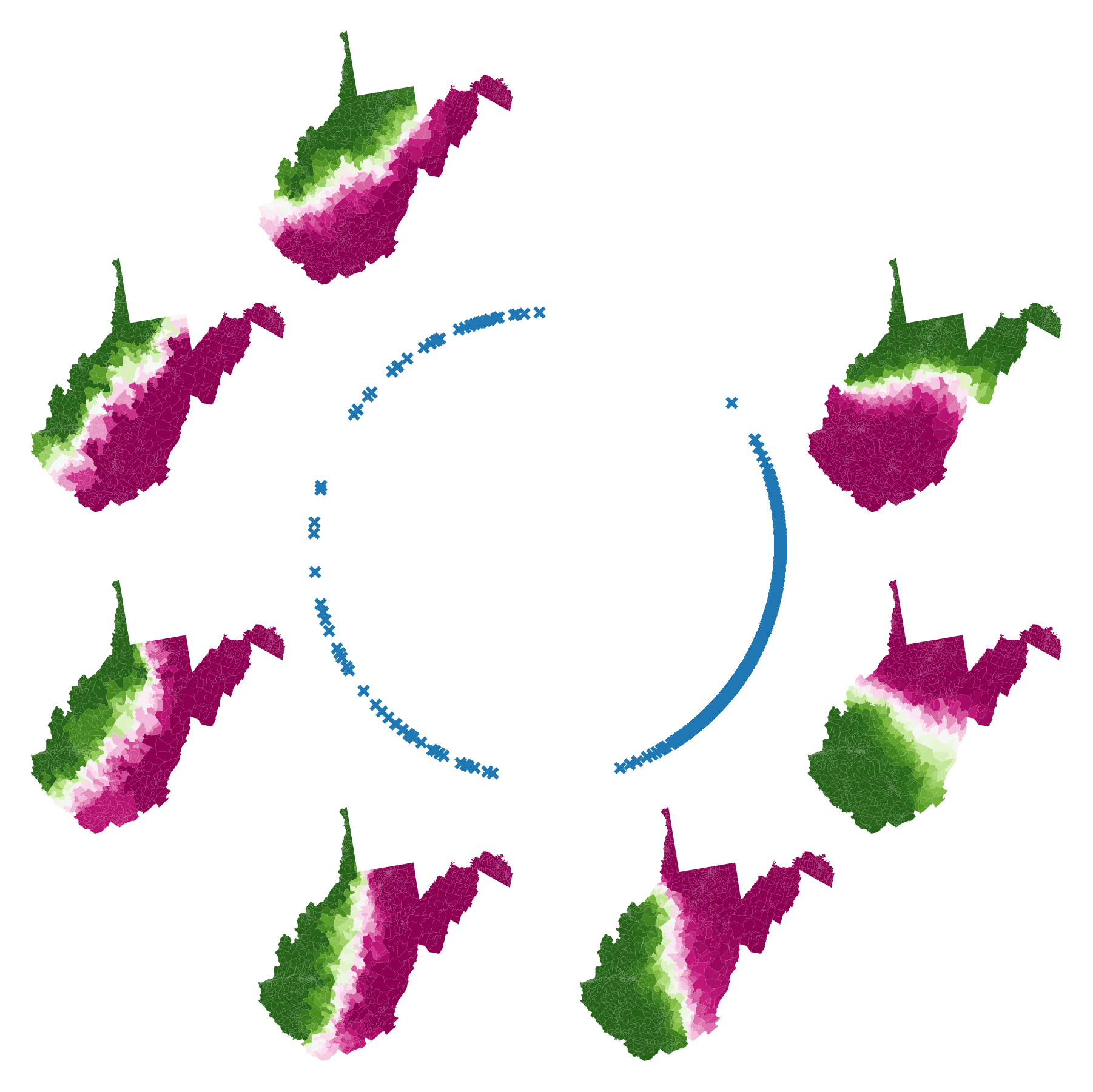}%
    \includegraphics[width=0.6\textwidth]{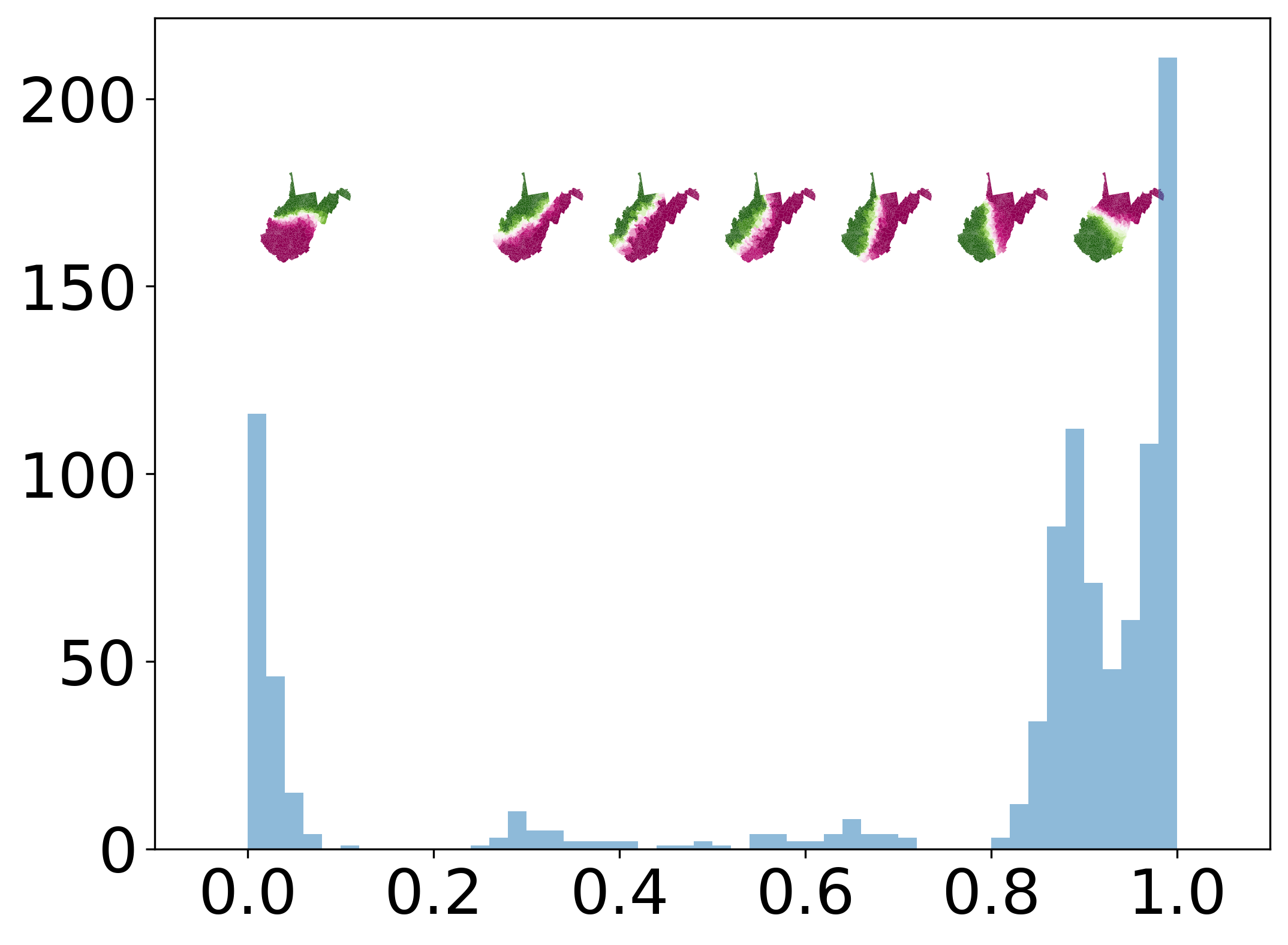}%
    \caption{West Virginia (WV)}
    \end{subfigure}

    \caption{Circle embeddings of 1,000-plan ensembles using SRGW+GD. Heat maps indicate the average district location for plans in each part of the circle.}\label{fig:circles}
\end{figure}

    \paragraph{Results.} In general, we see that circular coordinates roughly parameterize the angle of the boundary: from a north-south division, round to an east-west division and then back to a north-south division. This is most easily visible for {West Virginia} and {Montana}. This is strong evidence that the circle is a good choice of target space for embedding these ensembles. The boundary does not always look linear for states with a very uneven population distribution such as {Maine} (where most of the population is in the south of the state), though we still see a smooth rotation. 

    A general trend we can observe is a preference for divisions of the state with short (internal) boundary length. Boundary length is one possible measure of ``compactness'', a redistricting criterion often written into legislation around redistricting. The ReCom algorithm is known to favor low boundary lengths, where the boundary is measured by the number of Census blockgroups (or other geographic units) on the boundary of the districts~\cite{deford2021recombination, procaccia2022compact}. Preference for short boundaries can be observed in  {West Virginia}, where the northwest-southeast division requires a long boundary and is thus not likely to be drawn by the ReCom algorithm. In {Idaho}, we see a distribution with at least two modes: a northwest-southeast division and a northeast-southwest division. We should note that in {Idaho}, a straight, vertical boundary is so unlikely that it doesn't even show up on the heat maps; this is likely because most of the large boundary length this would require. In {Maine}, most plans are concentrated around northeast-southwest split. In New Hampshire most plans are concentrated around a roughly north-south split, though there is more variance than for Maine. These results for Maine and New Hampshire align with the analysis in \cite{asgari2020assessing} of redistricting in these states. In particular, the authors of \cite{asgari2020assessing} conclude that the enacted west-east redistricting plan in New Hampshire was an outlier compared to their ensemble, and propose that incumbent protection may have played a role in the drawing of that map. Our analysis also suggests that a east-west division of the state would be an outlier when compared to the ensemble.

\section*{Acknowledgments}

This material is based upon work supported by the National Science Foundation under the following grants. T.~Needham was supported by NSF grants DMS--2107808 and DMS--2324962. R.~A.~Clark was supported by the NSF MPS Ascending Postdoc Award 2138110. We would like to thank the anonymous reviewers for their helpful suggestions during the review process.

\bibliographystyle{plain}
\bibliography{bib}

\appendix

\section{Proofs}\label{app:proofs}

Before proceeding with the proofs, we briefly present a diagram summarizing the relationships between various dissimilarity measures considered in this paper.

\begin{figure}[h]
    \centering
    \begin{tikzpicture}
        \node at (0,0) {$\dgh$};
        \draw[->] (0,-0.5)--(0,-1.5);
        \node[anchor=west] at (0.1, -1)  {semi-relax};
        \node at (0,-2) {$\dsgh$};
        \draw[->] (0,-2.5)--(0,-3.5);
        \node[anchor=west] at (0.1, -3)  {symmetrize};
        \node at (0,-4) {$\symdsgh$};

        \node at (4,0) {};
        \node at (4,-2) {};
        \node at (4,-4) {$\modgh$};

        \node at (8,0) {$\dgwp$};
        \draw[->] (8,-0.5)--(8,-1.5);
        \node[anchor=west] at (8.1, -1)  {semi-relax};
        \draw[->] (8,-2.5)--(8,-3.5);
        \node[anchor=west] at (8.1, -3)  {symmetrize};
        \node at (8,-2) {$\dsgwp$};
        \node at (8,-4) {$\symdsgwp$};

        \draw[implies-implies,double equal sign distance] (8.6,-2)--(11.4, -2);
        \node at (12, -2) {MDS};
        \node at (10, -1.8) {Cor. \ref{cor:connection_to_MDS} $\ p=2$};

        \draw[double equal sign distance] (4.6,-4)--(7.4,-4);
        \node at (6, -3.8) {Th$^\text{m}$ \ref{thm:srgh_equals_srgw}$\ p=\infty$};
        
        \draw[double equal sign distance] (0.6,-4)--(3.4,-4);
        \node at (2, -3.8) {Th$^\text{m}$ \ref{thm:equivalent_mGH}};
    \end{tikzpicture}
    \caption{Summary of the relationships between various notions of dissimilarity defined in Sections~\ref{sec:gwdistances} and \ref{sec:gromov-type_distances}.}
    \label{fig:diagram}
\end{figure}

\subsection{Proof of Theorem \ref{thm:monge}}\label{sec:proof_of_monge}

The proof will use some preliminary definitions and a lemma. Recall that the \define{$p$-diameter}~\cite[Definition 5.2]{memoli2011gromov} of a mm-space $(X,d_X,\mu_X)$ is given by
\[
\mathrm{diam}_p(X)^p  = \int_X \int_X d_X(x,x')^p d\mu_X(x)d\mu_X(x') = \|d_X\|_{L^p(\mu_X \otimes \mu_X)}^p
\]
for $p \in [1,\infty)$, and, for $p=\infty$, the $p$-diameter $\|d_X\|_{L^\infty(\mu_X \otimes \mu_X)}$ is equal to the metric diameter of the support of $\mu_X$, which we denote simply as $\mathrm{diam}(X)$. For any metric space $(Y,d_Y)$ and any point $y_0 \in Y$, it is not hard to see that 
\begin{equation}\label{eqn:diameter_function}
\mathrm{diam}_p(X) = \dis_p(f_{y_0}),
\end{equation}
where $f_{y_0}:X \to Y$ is the function sending all of $X$ to $y_0$. For an arbitrary measurable map $f:X \to Y$, the \define{image} of $f$ is the mm-space $(\mathrm{image}(f),d_Y|_{\mathrm{image}(f)^2},f_\# \mu_X)$.

\begin{Lemma}\label{lem:monge}
    Let $(X,d_X,\mu_X)$ be a metric space with $X$ finite and $\mu_X$ fully supported and let $(Y,d_Y)$ be a proper metric space. Let $\gamma \in \semcoup(\mu_X,Y)$ be a semi-coupling from $X$ to $Y$ with finite $p$-distortion $\dis_p(\gamma)$. Then for any $p \in [1,\infty]$ there exists a function $f: X \to Y$ such that
    \[
    \dis_p(\mu_f) \leq \dis_p(\gamma),
    \]
    where $\mu_f$ is the semi-coupling induced by $f$. Moreover, the inequality is strict if $p < \infty$ and $\gamma$ is not itself induced by a function, and the image of $f$ can be chosen to have diameter satisfying 
    \begin{equation}\label{eqn:diameter_bound}
    \mathrm{diam}(\mathrm{image}(f)) \leq 2 \cdot \frac{\mathrm{diam}_p(X)}{\min_{x \in X} \{\mu_X(x)^{2/p}\}},
    \end{equation}
    for $p < \infty$. When $p = \infty$, the diameter bound is
    \[
    \mathrm{diam}(\mathrm{image}(f)) \leq 2 \cdot \mathrm{diam}(X).
    \]
\end{Lemma}
    
\begin{proof}
The case for $p = \infty$ follows from Theorems \ref{thm:equivalent_mGH} and \ref{thm:srgh_equals_srgw}---details are explained in Remark~\ref{rem:completing_monge_proof}. For the rest of the proof, we assume $p< \infty$. 

    For $x \in X$, let $\gamma_x$ be the measure on $Y$ coupled to $x$ -- that is, the pushforward along the second projection $\pi_2: X\times Y \to Y$ of the restriction of $\gamma$ to the subspace $\{x\} \times Y$. Said differently, $\gamma_x$ is the \emph{disintegration kernel} of $\gamma$ at $x$, and this measure satisfies 
    \begin{equation}\label{eqn:disintegration}
    \gamma = \sum_{x \in X} \mu_X(x) \cdot \delta_x \otimes \gamma_x. 
    \end{equation}
    The idea of the proof is to adjust $\gamma$ by replacing the disintegration kernels with Dirac masses in a manner which decreases $p$-distortion.
    
 Fix $x_0 \in X$ and define the function $g: Y \to \mathbb{R}$ by
\begin{align*}
   g(y)  &=  \int_{(X\setminus \{x_0\}) \times Y} |d_X(x_0,x') - d_Y(y,y')|^p d\gamma(x',y') \\
   &= \sum_{x' \in X \setminus \{x_0\}} \mu_X(x') \int |d_X(x_0,x') - d_Y(y, y')|^p \gamma_{x'}(y').
\end{align*}
Intuitively, $g(y)$ is the total distortion of distances from $x_0$, assuming that we couple $x_0$ to the point $y$. We claim that $g$ achieves a minimum in $Y$. If $Y$ is bounded, then it must be compact (as it is assumed to be proper), so the claim follows in this case by continuity of $g$. To handle the case where $Y$ is unbounded, first fix a basepoint $y_0 \in Y$. Since $\gamma$ has finite $p$-distortion, the moment
\[
m_x = \int d_Y(y_0, y)^p \gamma_x(y) 
\]
is finite for every $x \in X$. Next, consider the quantity
\[
A = \gamma((X \setminus \{x_0\}) \times Y).
\]
If $X = \{x_0\}$, then our lemma follows trivially, since any function would induce a coupling with zero distortion. Assuming that $X$ contains more than a single point, the assumption that $\mu_X$ is fully supported implies that $A$ is strictly positive. Using the inequality $\frac{|A+B|^p}{2^{p-1}} \leq |A|^p + |B|^p$ and the triangle inequality, we have
\begin{align*}
    g(y)  & = \sum_{x' \in X \setminus \{x_0\}} \mu_X(x') \int_Y |d_X(x_0,x') - d_Y(y, y')|^p d\gamma_{x'}(y')  \\
    & \geq \sum_{x' \in X \setminus \{x_0\}} \mu_X(x') \int_Y \frac{1}{2^{p-1}}d_Y(y, y')^p - d_X(x_0,x')^p d\gamma_{x'}(y') \\
    & \geq \sum_{x' \in X \setminus \{x_0\}} \mu_X(x') \int_Y \frac{1}{2^{p-1}}\left(\frac{1}{2^{p-1}} d_Y(y, y_0)^p - d_Y(y_0, y')^p \right) - d_X(x_0,x')^p d\gamma_{x'}(y') \\
    &= \frac{1}{2^{2p - 2}} d_Y(y,y_0)^p \sum_{x' \in X \setminus \{x_0\}} \mu_X(x') \int_Y d\gamma_{x'}(y) \\
    &\hspace{2in} - \frac{1}{2^{p-1}} \sum_{x' \in X \setminus \{x_0\}} \mu_X(x') \int_Y d(y_0,y')^p d\gamma_{x'}(y') \\
    &\hspace{2.5in} - \sum_{x' \in X \setminus \{x_0\}} \mu_X(x') \int_Y d_X(x_0,x')^p d\gamma_{x'}(y) \\
    &\geq \frac{1}{2^{2p-2}} d_Y(y,y_0)^p \cdot A - \sum_{x' \in X \setminus \{x_0\}} \frac{m_{x'}}{2^{p-1}} \mu_X(x') - \mathrm{diam}(X)^p \cdot A.
\end{align*}
Since this lower bound is strictly increasing in $d_Y(y,y_0)$ and we have assumed that $Y$ is unbounded, we can find a sufficiently large $M > 0$ such that if $d_Y(y,y_0) \geq M$ then $g(y) > g(y_0)$. As $Y$ is proper, the closed ball $B(y_0, M)$ is compact, so continuity of $g$ implies that $g$ achieves a minimum on this ball at some $\hat{y} \in B(y_0, M)$. This minimum is therefore a global minimum for $g$. 

Now consider the measure
\[
\tilde{\gamma} = \mu_X(x_0) \delta_{(x_0, \hat{y})} + \sum_{x \in X \setminus \{x_0\}} \mu_X(x) \cdot \delta_x \otimes \gamma_{x}.
\]
That is, $\tilde{\gamma}$ is obtained from $\gamma$ by replacing the disintegration kernel $\gamma_{x_0}$ with the weighted Dirac mass $\mu_X(x_0)\delta_{\hat{y}}$ (see \eqref{eqn:disintegration}). It is then easy to see that $\tilde{\gamma} \in \semcoup(\mu_X,Y)$. We will show that
\[
\dis_p(\tilde{\gamma}) \leq \dis_p(\gamma).
\]
To do so, we define the following quantity measuring the distortion induced by an arbitrary semicoupling $\alpha \in \semcoup(\mu_X,Y)$ of the distance between a given pair of points $x,x' \in X$:
\[
C_{x,x'}(\alpha) = \int_Y \int_Y |d_X(x,x') - d_Y(y, y')|^p d\alpha_x(y) d\alpha_{x'}(y')
\]
(with $\alpha_x$ denoting the disintegration kernel of $\alpha$ at $x$). Then the $p$-distortion of $\gamma$ satisfies
\begin{multline}\label{eq:fulldis}
        \dis_p(\gamma)^p =  C_{x_0,x_0}(\gamma) \mu_X(x_0)^2 + 2 \mu_X(x_0) \sum_{x' \in X \setminus \{x_0\}} C_{x_0,x'}(\gamma) \mu_X(x') + \sum_{x, x' \in X \setminus \{x_0\}} C_{x,x'}(\gamma) \mu_X(x) \mu_X(x').
\end{multline}
Consider the corresponding terms in $\dis_p(\tilde{\gamma})^p$. It is straightforward to see that
\[
C_{x_0,x_0}(\tilde{\gamma}) = 0 \quad\mbox{and}\quad  \sum_{x, x' \in X \setminus \{x_0\}} C_{x,x'}(\tilde{\gamma}) \mu_X(x) \mu_X(x') =  \sum_{x, x' \in X \setminus \{x_0\}} C_{x,x'}(\gamma) \mu_X(x) \mu_X(x').
\]
Moreover, we have
\begin{align*}
    &\sum_{x' \in X \setminus \{x_0\}} C_{x_0, x'}(\gamma) \mu_X(x') \\
    &\qquad = \sum_{x' \in X \setminus \{x_0\}} \mu_X(x') \int_Y \int_Y |d_X(x_0,x') - d_Y(y, y')|^p d\gamma_{x_0}(y) d\gamma_{x'}(y') \\
    &\qquad = \int_Y \left( \sum_{x' \in X \setminus \{x_0\}} \mu_X(x')  \int_Y |d_X(x_0,x') - d_Y(y, y')|^p  d\gamma_{x'}(y') \right) d\gamma_{x_0}(y) \\
    &\qquad = \int_Y g(y) d\gamma_{x_0}(y) \\
    &\qquad\geq  g(\hat{y}) \mu_X(x_0) \\
    &\qquad= \sum_{x' \in X \setminus \{x_0\}} \mu_X(x_0) \int_Y |d_X(x_0,x') - d_Y(\hat{y},y')|^p \mu_X(x_0) d\tilde{\gamma}_{x'}(y') \\
    &\qquad= \sum_{x' \in X \setminus \{x_0\}} \mu_X(x_0) \int_Y \int_Y |d_X(x_0,x') - d_Y(y,y')|^p d\tilde{\gamma}_{x_0}(y) d\tilde{\gamma}_{x'}(y') \\
    &\qquad= \sum_{x' \in X \setminus \{x_0\}} C_{x_0,x'}(\tilde{\gamma}) \mu_X(x').
\end{align*}
Putting this together, we have
\[
\dis_p(\tilde{\gamma}) \leq \dis_p(\gamma),
\]
with a strict inequality if $x_0$ was not already coupled to a unique point. Performing the same replacement for each point in $X$ yields a function $f$ such that $\dis_p(\mu_f) \leq \dis_p(\gamma)$, with strict inequality if $\gamma$ is not already induced by a function.

To prove the diameter bound \eqref{eqn:diameter_bound}, first suppose that the image of a function $f$, $\mathrm{image}(f)$, has $p$-diameter strictly larger than $2 \cdot \mathrm{diam}_p(X)$. Consider the function $f_{y_0}:X \to Y$ with image a single point $y_0$. Then
\begin{align}
\dis_p(\mu_f) &= \|d_X - d_Y\|_{L^p(\mu_f \otimes \mu_f)} = \|d_X - d_Y \circ (f \times f)\|_{L^p(\mu_X \otimes \mu_X)} \nonumber \\
&\geq \left|\|d_X\|_{L^p(\mu_X \otimes \mu_X)} - \|d_Y \circ (f \times f)\|_{L^p(\mu_X \otimes \mu_X)}\right| = |\mathrm{diam}_p(X) - \mathrm{diam}_p(\mathrm{image}(f))| \nonumber \\
&> \mathrm{diam}_p(X) = \mathrm{dis}_p(\mu_{f_{y_0}}), \label{eqn:diameter_bound_1}
\end{align}
where the last line follows by \eqref{eqn:diameter_function}. One can therefore replace $f$ by $f_{y_0}$ while achieving smaller $p$-distortion, and it follows that we can assume that $\mathrm{diam}_p(\mathrm{image}(f)) \leq 2 \cdot \mathrm{diam}_p(X)$. To complete the proof, we write $X = \{x_1,\ldots,x_N\}$ and observe that
\begin{align*}
    \mathrm{diam}_p(\mathrm{image}(f)) &= \left(\sum_{i,j=1^N} d_Y(f(x_i),f(x_j))^p \mu_X(x_i) \mu_X(x_j)\right)^{1/p} \\
    &\geq \max_{i,j} d_Y(f(x_i),f(x_j))\mu_X(x_i)^{1/p} \mu_X(x_j)^{1/p} \\
    &\geq \max_{i,j} d_Y(f(x_i),f(x_j)) \cdot \min_k \mu_X(x_k)^{2/p} \\
    &= \mathrm{diam}(\mathrm{image}(f)) \cdot \min_k \mu_X(x_k)^{2/p},
\end{align*}
and \eqref{eqn:diameter_bound} follows.
\end{proof}

\begin{proof}[Proof of Theorem \ref{thm:monge}]
Let $K$ be a compact subset of $Y$ such that the translates of $Y$ by $G$ have Lebesgue number $R$ with 
\[
R > \left\{\begin{array}{rl}
2\cdot \mathrm{diam}_p(X) /\min_{x \in X} \{\mu_X(x)^{2/p}\} & \mbox{if } p < \infty \\
2 \cdot \mathrm{diam}(X) & \mbox{if } p = \infty.
\end{array}\right.
\]
Consider the problem of computing $\dsgwp(X, K)$. Since $K$ is compact, there is an optimal semi-coupling of $X$ to $K$ induced by a map $f: X \to Y$. Indeed, by Lemma \ref{lem:monge}, it is sufficient to optimize over the space $K^X$ of functions from $X$ to $K$, which is compact, and $f \mapsto \dis_p(\mu_f)$ is continuous. We claim that the map $f: X \to Y$ given by composing with the inclusion $K \subseteq Y$ suffices to prove the theorem. Suppose there exist some semi-coupling $\gamma \in \semcoup(\mu_X,Y)$ with lower $p$-distortion than $\mu_f$. By Lemma \ref{lem:monge}, we may assume that $\gamma$ is itself induced by a map $g: X \to Y$ whose image has diameter less than $R$. By construction, the image of $g$ is contained in a translate $hK$ of $K$ by some $h \in G$. The map $h^{-1}g: X \to Y$ has the same distortion as $g$ since the action of $h$ is an isometry, and its image is contained in $K$. But since $f$ was assumed to give an optimal coupling from $X$ to $K$, this is a contradiction. 
\end{proof}

\subsection{Proof of Theorem \ref{thm:equivalent_mGH}}\label{sec:proof_equivalent_mGH}

\begin{proof}
    Let $X$ and $Y$ be metric spaces. For any map $f:X \to Y$, define $R_f = \{(x,f(x)) \mid x \in X\}$. Then $R_f$ is a semi-correspondence, and 
    \begin{align*}
    \dis(R_f) &= \frac{1}{2} \sup_{(x,y),(x',y') \in R_f} |d_X(x,x')-d_Y(y,y')| \\
    &= \frac{1}{2} \sup_{x,x' \in X} |d_X(x,x')-d_Y(f(x),f(x'))| = \dis(f).
    \end{align*}
    Therefore,
    $
    \dsgh(X,Y) \leq \inf_{f:X \to Y} \dis(f).
    $
    Define $R_g$ similarly for a map $g:Y \to X$, so that 
    $
    \dsgh(Y,X) \leq  \inf_{g: Y \to X} \dis(g).
    $
    It follows that 
    \begin{align*}
    \symdsgh(X,Y) &= \max \{\dsgh(X,Y), \dsgh(Y,X)\} \\
    &\leq \max \left\{\inf_{f:X \to Y} \dis(f), \inf_{g:Y \to X} \dis(g)\right\} = \modgh(X,Y).
    \end{align*}

    To prove the other inequality, let $R \in \semcorr(X,Y)$. Define a map $f_R:X \to Y$ by choosing, for each $x \in X$, some $y \in Y$ such that $(x,y) \in R$, and setting $f_R(x) = y$. Then
    \begin{align*}
    \mathrm{dis}(f_R) &= \sup_{x,x' \in X} |d_X(x,x') - d_Y(f_R(x),f_R(x'))| \\
    &\leq \sup_{(x,y),(x',y') \in R} |d_X(x,x') - d_Y(y,y')| = \dis(R).
    \end{align*}
    By similar reasoning to above, it follows that $\modgh(X,Y) \leq \symdsgh(X,Y)$.
\end{proof}

\subsection{Proof of Theorem \ref{thm:srgh_equals_srgw}}\label{sec:proof_srgh_equals_srgw}

The proof will use a lemma.

\begin{Lemma}\label{lem:correspondence_approximation}
     Let $X$ and $Y$ be mm-spaces such that $\mu_X$ has full support and let  $R \in \mathcal{SR}(X,Y)$. For any $\epsilon > 0$, there exists $R_\epsilon \in \mathcal{SR}(X,Y)$ such that $R_\epsilon$ is a Borel measurable set and $|\mathrm{dis}(R) - \mathrm{dis}(R_\epsilon)| < \epsilon$.
\end{Lemma}

\begin{proof}
    For any $p = (x,y) \in R$, let $U_p = B_X(x,\epsilon/4) \times B_Y(y,\epsilon/4)$ (with $B_\bullet(\cdot,\cdot)$ denoting an open metric ball in the appropriate space), and set $R_\epsilon = \bigcup_{p \in R} U_p$. Then $R_\epsilon$ is a Borel set, and we can bound the difference in distortions as follows. For any $(x,y),(x',y') \in R_\epsilon$, choose $p_0 = (x_0,y_0)$ and $p_0' = (x_0',y_0')$ in $R$ such that $(x,y) \in U_{p_0}$ and $(x',y') \in U_{p_0'}$. Then
    \begin{multline}\label{eqn:srgh_equals_srgw_triangle_ineq}
    |d_X(x,x') - d_Y(y,y')| \leq |d_X(x,x') - d_X(x_0,x_0')| \\ + |d_X(x_0,x_0') - d_Y(y_0,y_0')| + |d_Y(y_0,y_0') - d_Y(y,y')|.
    \end{multline}
    The first term on the right hand side of \eqref{eqn:srgh_equals_srgw_triangle_ineq} is bounded above by $\epsilon/2$. To see this, first observe that if $d_X(x,x') \geq d_X(x_0,x_0')$ then this term is bounded as
    \[
    d_X(x,x') - d_X(x_0,x_0') \leq d_X(x,x_0) + d_X(x_0,x_0') + d_X(x_0',x') - d_X(x_0,x_0') < 2 \cdot \frac{\epsilon}{4}.
    \]
    The case where $d_X(x,x') \leq d_X(x_0,x_0')$ follows similarly. Likewise, the last term on the right hand side of \eqref{eqn:srgh_equals_srgw_triangle_ineq} is bounded above by $\epsilon/2$. Putting this together with \eqref{eqn:srgh_equals_srgw_triangle_ineq}, we have 
    \[
    |d_X(x,x') - d_Y(y,y')| < |d_X(x_0,x_0') - d_Y(y_0,y_0')| + \epsilon \leq \mathrm{dis}(R) + \epsilon.
    \]
    Since $(x,y),(x',y') \in R_\epsilon$ were arbitrary, this implies 
    \[
    |\mathrm{dis}(R_\epsilon) - \mathrm{dis}(R)| = \mathrm{dis}(R_\epsilon) - \mathrm{dis}(R) < \epsilon,
    \]
    where the equality follows because $R \subset R_\epsilon$. This proves the claim.
\end{proof}

\begin{proof}[Proof of Theorem \ref{thm:srgh_equals_srgw}]
    We will show that the unsymmetrized versions satisfy $d_{\mathrm{srGW},\infty} = d_{\mathrm{srGH}}$, from which the claim follows. Observe that 
    \[
    d_{\mathrm{srGW},\infty}(X,Y) = \inf_{\gamma \in \mathcal{SC}(\mu_X,Y)} \mathrm{dis}(\mathrm{supp}(\gamma)),
    \]
    where $\mathrm{dis}$ is the same distortion that appears in the definition of (semi-relaxed) Gromov-Hausdorff distance. Since $\mu_X$ is assumed to be fully supported, we have 
    \[
    \{\mathrm{supp}(\gamma) \mid \gamma \in \mathcal{SC}(\mu_X,Y)\} \subset \mathcal{SR}(X,Y)
    \]
    and it follows that $d_{\mathrm{srGH}}(X,Y) \leq d_{\mathrm{srGW},\infty}(X,Y)$. 
    
    To prove the reverse inequality, let $R \in \mathcal{SR}(X,Y)$. For $\epsilon > 0$, let $R_\epsilon$ be as in Lemma \ref{lem:correspondence_approximation} and define a  measure $\gamma$ on $X \times Y$ by defining it on a Borel set $V \subset X \times Y$ as
    \[
    \gamma(V) = \mu_X\left(\mathrm{proj}_X\big(V \cap R_\epsilon \big) \right).
    \]
    This is well-defined, since $\mathrm{proj}_X\big(V \cap R_\epsilon\big)$ must be a Borel set in $X$. Moreover, for any measurable $A \subset X$, we have
    \[
    \gamma(A \times Y) = \mu_X\left(\mathrm{proj_X}\big((A \times Y) \cap R_\epsilon\big)\right) = \mu_X(A),
    \]
    so that $\gamma$ has the correct marginal condition. Finally, the support of $\gamma$ is $R_\epsilon$, so that $|\mathrm{dis}(R) - \mathrm{dis}(\mathrm{supp}(\gamma))| < \epsilon$. Since this holds for arbitrary $R$, we have $|d_{\mathrm{srGW},\infty}(X,Y) - d_\mathrm{srGH}(X,Y)| < \epsilon$ for any $\epsilon$, and the claim follows.

    It remains to prove the last statement of the theorem, on the metric properties of $\dsgwp$. That $\widehat{d}_{\mathrm{srGW},\infty}$ defines a metric with the given properties follows from the corresponding properties of $\modgh$, proved in Theorems 4.1 and 4.2 from ~\cite{memoli2012some}. We will prove that $\widehat{d}_{\mathrm{srGW},p}$ is not a pseudometric for $p < \infty$ by counterexample. Let $(X,d_X,\mu_X)$ be the one point mm-space, let $X = Y = \{a,b\}$ with $d_X(a,b) = d_Y(a,b) = 1$, but with different measures defined by $\mu_X(a) = \mu_X(b) = 1/2$, and $\mu_Y(a) = 1-\epsilon$ and $\mu_Y(b) = \epsilon$ for some $\epsilon > 0$, to be determined. Let $(Z,d_Z,\mu_Z)$ be the one-point metric measure space. Then it is not hard to show that
    \[
    \symdsgwp(X,Z) = \left(\frac{1}{2}\right)^{1/p}, \quad \symdsgwp(X,Y) = 0, \quad \symdsgwp(Y,Z) = (2\epsilon(1-\epsilon))^{1/p},
    \]
    so that 
    \[\symdsgwp(X,Z) > \symdsgwp(X,Y) + \symdsgwp(Y,Z)\]
    holds whenever $p < \infty$ and $\epsilon \neq 1/2$. Thus the triangle inequality fails when $p < \infty$. We note that similar counterexamples can be used to show that the triangle inequality even fails for the unsymmetrized version of srGW distance when $p < \infty$. 
\end{proof}

\begin{Remark}[Completing the Proof of Lemma~\ref{lem:monge}]\label{rem:completing_monge_proof}
    Recall that the proof of the $p=\infty$ case of Lemma~\ref{lem:monge} was claimed to follow from Theorems \ref{thm:equivalent_mGH} and \ref{thm:srgh_equals_srgw}. We now fill in those details and complete the proof of the lemma. 

    Let $p = \infty$ and suppose that $\gamma \in \semcoup(\mu_X,Y)$ with $\mathrm{dis}_\infty(\gamma) < \infty$. Observe that $\mathrm{dis}(\gamma) = \mathrm{dis}(R)$, where $R:=\mathrm{supp}(\gamma)$ and $\mathrm{dis}$ is the Gromov-Hausdorff distortion. Following the proof of Theorem \ref{thm:equivalent_mGH}, we can choose a function $f_R:X \to Y$ with $\mathrm{dis}(f_R) \leq \mathrm{dis}(R)$, and this implies $\mathrm{dis}_\infty(\mu_{f_R}) \leq \mathrm{dis}_\infty(\gamma)$. This completes the proof of the first part of the lemma.

    It remains to prove the diameter bound. Recall the argument given in \eqref{eqn:diameter_bound_1} to show that $\mathrm{diam}(\mathrm{image}(f)) < 2 \cdot \mathrm{diam}_p(X)$. This argument is written entirely in terms of $L^p$ norms, and still works in the $p = \infty$ case. This gives the desired diameter bound immediately.
\end{Remark}

\section{An Embedding Formulation}\label{sec:embedding_formulation}

Let $Z = (Z,d_Z)$ be a metric space and let $X,Y \subset Z$. The \define{Hausdorff distance} between $X$ and $Y$ is given by
\[
\dhaus^Z(X,Y) = \max\left\{ \sup_{x \in X} \inf_{y \in Y} d_Z(x,y), \sup_{y \in Y} \inf_{x \in X} d_Z(x,y) \right\}.
\]
Another equivalent formulation of Gromov-Hausdorff distance is given by
\begin{equation}\label{eqn:embedding_GH}
\dgh(X,Y) = \inf_{f,g,Z} \dhaus^Z(f(X),g(Y)),
\end{equation}
where the infimum is over metric spaces $Z$ and isometric embeddings $f:X \to Z$ and $g:Y \to Z$. We now give a similar reformulation of the semi-Relaxed Gromov-Hausdorff distance.

\begin{Definition}[Semi-Relaxed Hausdorff Distance]
    Let $Z$ be a metric space and let $X,Y \subset Z$. We define the \define{semi-relaxed Hausdorff (srH) distance} between $X$ and $Y$ to be
    \[
    \dsh^Z(X,Y) = \inf_{R \in \semcorr(X,Y)} \sup_{(x,y) \in R} d_Z(x,y).
    \]
\end{Definition}

The following proposition says that the srH distance is really the same as Hausdorff distance, without the symmetrization (this is sometimes referred to as the asymmetric Hausdorff distance). 

\begin{Proposition}\label{prop:Hausdorff_distance}
    Let $Z$ be a metric space and let $X,Y \subset Z$. Then the Hausdorff distance between $X$ and $Y$ is given by
    \[
    d_\mathrm{H}^Z(X,Y) = \max\{\dsh^Z(X,Y), \dsh^Z(Y,X)\}.
    \]
\end{Proposition}

\begin{proof}
    We wish to show that
    \[
    \dsh^Z(X,Y) = \sup_{x \in X} \inf_{y \in Y} d_Z(x,y).
    \]
    First suppose that, for every $x \in X$, there exists $y \in Y$ such that $d_Z(x,y) \leq \epsilon$. Define
    \[
    R = \{(x,y) \mid d_Z(x,y) \leq \epsilon\}.
    \]
    Then $R$ is a semi-correspondence and $\sup_{(x,y) \in R} d_Z(x,y) \leq \epsilon$, by construction. This proves that 
    $$\dsh^Z(X,Y) \leq \sup_{x \in X} \inf_{y \in Y} d_Z(x,y).$$
    On the other hand, suppose that $R \in \semcorr(X,Y)$ such that $d_Z(x,y) \leq \epsilon$ for all $(x,y) \in R$. Then, for every $x \in X$, there exists $y \in Y$ such that $d_Z(x,y) \leq \epsilon$---namely, choose any $y$ such that $(x,y) \in R$. This completes the proof.
\end{proof}

\begin{Theorem}\label{thm:embedding_formuation_GH}
    For metric spaces $X$ and $Y$, we have
    \[
    \dsgh(X,Y) = \inf_{f,g,Z} \dsh^Z(f(X),g(Y)),
    \]
    where the infimum is over metric spaces $Z$ and isometric embeddings $f:X \to Z$ and $g:Y \to Z$.
\end{Theorem}

\begin{proof}
    Let $Z$ be a metric space and $f:X \to Z$ and $g:Y \to Z$ isometric embeddings such that $\dsh^Z(f(X),g(Y)) \leq \epsilon$. To simplify notation, we can assume without loss of generality that $X$ and $Y$ are actually subsets of $Z$, that $d_X = d_Z|_{X \times X}$ and $d_Y = d_Z|_{Y \times Y}$, and that $\dsh^Z(X,Y) \leq \epsilon$. Let 
    \[
    R = \{(x,y) \in X \times Y \mid d_Z(x,y) \leq \epsilon\}. 
    \]
    Then $R$ is a semi-coupling and, for any $(x,y),(x',y') \in R$, we have
    \begin{align*}
        |d_X(x,x') - d_Y(y,y')| &= |d_Z(x,x') - d_Z(y,y')| \\
        &\leq d_Z(x,y) + d_Z(x',y') \leq 2\epsilon,
    \end{align*}
    so that $\dis(R) \leq \epsilon$. This shows that $\dsgh(X,Y) \leq \inf_{f,g,Z} \dsh^Z(f(X),g(Y))$.

    Now suppose that $R$ is a semi-coupling with $\dis(R) \leq \epsilon$. Put a metric $d_Z$ on the space $Z = X \sqcup Y$ which is defined by $d_Z|_{X \times X} = d_X$, $d_Z|_{Y \times Y} = d_Y$ and 
    \[
    d_Z(x,y) = \inf_{(x',y') \in R} d_X(x,x') + d_Y(y,y') + \epsilon
    \]
    when $x \in X$ and $y \in Y$ (with $d_Z(y,x)$ defined similarly). It is straightforward to check that $d_Z$ is really a metric and we have, for $f:X \to Z$ and $g:Y \to Z$ the inclusion maps,
    \[
    \dsh^Z(f(X),g(Y)) \leq \sup_{(x,y) \in R} d_Z(x,y) = d_X(x,x) + d_Y(y,y) + \epsilon = \epsilon.
    \]
    This shows that $\dsgh(X,Y) \geq \inf_{f,g,Z} \dsh^Z(f(X),g(Y))$, so the proof is complete.
\end{proof}

\begin{Remark}
    By Theorem \ref{thm:embedding_formuation_GH}, the symmetrized semi-relaxed Gromov-Hausdorff distance is given by
    \[
    \symdsgh(X,Y) = \max \left\{\inf_{f,g,Z} \dsh^Z(f(X),g(Y)), \inf_{h,k,W} \dsh^W(h(Y),k(X))\right\}.
    \]
    This formulation of the distance compares interestingly with the embedding formulation of GH distance \eqref{eqn:embedding_GH}, which, by Proposition \ref{prop:Hausdorff_distance}, can be written as 
    \[
    \dgh(X,Y) = \inf_{f,g,Z} \max \{\dsh^Z(f(X),g(Y)), \dsh^Z(g(Y),f(X))\}.
    \]
\end{Remark}

\section{Computational details} \label{app:computing}

\subsection{Pseudocode for SRGW+GD}

See Algorithm~\ref{alg:sgrwgd} for pseudocode outlining the SRGW+GD algorithm. Note that in some cases, the metric on $Y$ may take a scale parameter, which will also be updated by gradient descent on $F$.

\begin{algorithm}
\caption{SRGW+GD}\label{alg:sgrwgd}
\begin{algorithmic}
\Require finite metric space $(X, d)$ (as a distance matrix) with $X = \{x_1, \ldots, x_n\}$
\Require discrete finite subset of target space $\hat{Y} = \{\hat{y}_1,\ldots, \hat{y}_m\} \subseteq Y$
\Require learning rate $\alpha$ 

\\

\State $ \gamma \gets $ optimal semi-coupling from $X$ to $\hat{Y}$ \Comment Equation (\ref{eqn:srGW})

\State $\forall_{1 \leq i \leq n} \ y_i \gets \hat{y}_j \ \mathrm{s.t. } \ \gamma(x_i, \hat{y}_j) > 0$ \Comment{Monge map on $x_i$}

\State $\forall_{1 \leq i \leq n} \ y_i \gets y_i + \text{small noise}$ \Comment optional step for better gradients

\\

\State let $F(\mathbf{y}) = F(y_1, \ldots, y_n) := \sum_{i,j = 1}^n (d_X(x_i, x_j) - d_Y(y_i, y_j))^2$

\Repeat 
\State $\mathbf{y} \gets \mathbf{y} - \alpha \cdot \nabla F$ \Comment gradient descent
\State $\alpha \gets  \mathrm{update}(\alpha)$ \Comment optional Adam learning rate update
\Until{convergence}

\\

\State \Return $y_1,\ldots, y_n$ 
\end{algorithmic}
\end{algorithm}

\subsection{Hyperparameters}
The two main hyperparameters for the SRGW+GD experiments are the points in the subset $S \subseteq Y$ used in the initial embedding, and the learning rate $\gamma$ for the Adam optimizer. When embedding into $\mathbb{R}^2$ with SRGW+GD, we used an evenly spaced $20\times 20$ grid whose width was twice the diameter of the input space. We chose a set of $|S| = 100$ evenly spaces points when embedding R-MNIST into $S^1$, and $|S| = 1000$ for the redistricting experiments. For the Adam optimizer, we use a learning rate of $\gamma= 0.01$ for the MNIST and redistricting experiments and $\gamma = 0.1$ for the Cities dataset, chosen to achieve good convergence. We use a relative convergence threshold of $10^{-4}$ on the loss function for SMACOF MDS and SRGW+GD. The convergence behavior for GD was harder to control, so we used a $10^{-3}$ relative convergence threshold.  

\subsection{Hardware and compute times}
The experiments in Section \ref{sec:experiments} and \ref{sec:redistricting} were performed on a 2019 MacBook Pro with 64 GB of RAM. To give an idea of the compute time involved, we report the times taken to compute the embeddings from Figure \ref{fig:scatterplots} in Supplementary Table \ref{tab:times} below. The redistricting embeddings in \ref{fig:circles} took only a few seconds each to compute.

\begin{table}[h]
    \centering
    \begin{tabular}{l|r}
        method & time (s) \\ \hline \hline
      t-SNE   & 17.367 \\ \hline
      PCA   & 0.188\\ \hline
      SMACOF MDS   & 1061.073\\ \hline
      SRGW+GD $\mathbb{R}^2$ & 193.869 \\ \hline
      CC   & 50.287\\ \hline
      GD   & 14.572 \\ \hline
      SRGW+GD $S^1$ & 137.237 \\
    \end{tabular}
    \caption{Compute times for Figure \ref{fig:scatterplots}}.
    \label{tab:times}
\end{table}

For Figure \ref{fig:scatterplots} and the versions of it below, we use the first trial of the GD runs.

\section{Details on redistricting ensemble plots}\label{app:circles}

To create the visualizations in Figure \ref{fig:circles}, we divide the circle into eight arcs of equal length, with the first arc starting at $(1,0)$ and extending counterclockwise. Assigned to each arc is a subset of the ensemble. We refer to the plan with the lowest circular coordinate in each arc as the \emph{first plan} in that arc.

\begin{Definition}
    Consider two redistricting plans $P$ and $Q$, where $P$ has labelled districts. To \emph{align} $Q$ to $P$ means to choose labels for the districts in $Q$ such that the the number of Census blockgroups whose district number must be changed to get from $P$ from $Q$ is minimized. 
\end{Definition}

To set up our visualization, we arbitrarily label the first plan in the first arc, and align all other plans in the first arc to it. For each subsequent arc, we align the first plan in that arc to the first plan in the previous arc, and then align all the other plans in the arc to the first plan in the ensemble. This gives us a coherent labelling of every plan. Finally, we summarize all the redistricting plans in an arc with a heatmap showing, for each precinct $p$, the fraction of plans which assigned $p$ to District 1. The results are shown in Figure \ref{fig:circles} for all six states, with the heatmaps arranged near the arcs they summarize. 

\FloatBarrier

\section{Embeddings of other digits}\label{app:otherdigits}

The results of the rotated MNIST digit experiment for the remaining digits are shown in Supplementary Figures \ref{fig:map0}--\ref{fig:map8}.

\foreach \n in {0,1,2,3,4,5,6,7,8}
{
\begin{figure}
\centering
\resizebox{0.95\textwidth}{!}{
    \begin{tikzpicture}[scale=1.0]
        \node[rotate=90] at (-1.1,-2.5) {\tiny angular coord.};
        \node at (0.1,-3.6) {\tiny rotation angle};
        \node at (13.7,-1.5) {\includegraphics[width=1.3cm]{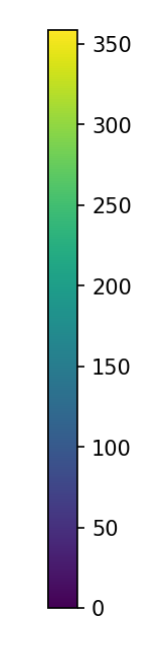}};
        
        \node at (0,1) {t-SNE};
        \node at (0,-0.3) 
        {\includegraphics[height=2cm]{MNIST/MNIST\n_TSNE.png}};
        \node at (0,-2.5) 
        {\includegraphics[height=2cm]{MNIST/MNIST\n_scatter_TSNE.png}};
        
        \node at (2,1) {PCA};
        \node at (2,-0.3) 
        {\includegraphics[height=2cm]{MNIST/MNIST\n_PCA.png}};
        \node at (2,-2.5) {\includegraphics[height=2cm]{MNIST/MNIST\n_scatter_PCA.png}};
        
        \node at (4,1) {MDS};
        \node at (4,-0.3) 
        {\includegraphics[height=2cm]{MNIST/MNIST\n_MDS.png}};
        \node at (4,-2.5) {\includegraphics[height=2cm]{MNIST/MNIST\n_scatter_MDS.png}};

        \node at (6,1) {SRGW+GD};
        \node at (6,-0.3) 
        {\includegraphics[height=2cm]{MNIST/MNIST\n_SRGW_R2.png}};
        \node at (6,-2.5) {\includegraphics[height=2cm]{MNIST/MNIST\n_scatter_SRGW_R2.png}};
        
        \node at (8,1) {CC};
        \node at (8,-0.3) 
        {\includegraphics[height=2cm]{MNIST/MNIST\n_PCOH.png}};
        \node at (8,-2.5) {\includegraphics[height=2cm]{MNIST/MNIST\n_scatter_PCOH.png}};
        
        \node at (10,1) {GD};
        \node at (10,-0.3) 
        {\includegraphics[height=2cm]{MNIST/MNIST\n_GD.png}};
        \node at (10,-2.5) {\includegraphics[height=2cm]{MNIST/MNIST\n_scatter_GD.png}};

        \node at (12,1) {SRGW+GD};
        \node at (12,-0.3) 
        {\includegraphics[height=2cm]{MNIST/MNIST\n_SRGW.png}};
        \node at (12,-2.5) {\includegraphics[height=2cm]{MNIST/MNIST\n_scatter_SRGW.png}};
    \end{tikzpicture}
    }
    \caption{Circle and planar embeddings for images of \n s in the form of Figure \ref{fig:scatterplots}.}\label{fig:map\n}
\end{figure}
}

\FloatBarrier

\section{Other embedding methods on redistricting data}

We embed the redistricting ensembles using various other non-linear dimension reduction methods, including \textbf{Laplacian Eigenmaps} with default \texttt{sklearn} settings and \textbf{Isomap}.

\foreach \state in {ID, ME, MT, NH, RI, WV}
{
\begin{figure}[ht]
    \centering
    \begin{subfigure}{0.16\textwidth}
    \centering
        \includegraphics[height=2.2cm, width=2cm, keepaspectratio]{redistricting/ensemble_\state_SRGW.png}
        \caption*{SRGW+GD}
    \end{subfigure}
    \begin{subfigure}{0.16\textwidth}
    \centering
        \includegraphics[height=2.2cm, width=2cm, keepaspectratio]{redistricting/ensemble_\state_MDS.png}
        \caption*{MDS}
    \end{subfigure}
    \begin{subfigure}{0.16\textwidth}
    \centering
        \includegraphics[height=2.2cm, width=2cm, keepaspectratio]{redistricting/ensemble_\state_Laplacian.png}
        \caption*{Laplacian}
    \end{subfigure}
    \begin{subfigure}{0.16\textwidth}
    \centering
        \includegraphics[height=2.2cm, width=2cm, keepaspectratio]{redistricting/ensemble_\state_TSNE.png}
        \caption*{t-SNE}
    \end{subfigure}
    \begin{subfigure}{0.16\textwidth}
    \centering
        \includegraphics[height=2.2cm, width=2cm, keepaspectratio]{redistricting/ensemble_\state_Isomap.png}
        \caption*{Isomap}
    \end{subfigure}
    \begin{subfigure}{0.16\textwidth}
    \centering
        \includegraphics[height=2.2cm, width=2cm, keepaspectratio]{redistricting/ensemble_\state_PCOH.png}
        \caption*{CC}
    \end{subfigure}
    \caption{Different embeddings for a 1,000 plan ensemble on \state. Color indicates the circular coordinate as determined by SRGW+GD.}
    \label{fig:others}
\end{figure}
}

\end{document}